\documentclass[12pt]{article}
\usepackage[english]{babel}
\usepackage{graphicx}
\usepackage{framed}
\usepackage[normalem]{ulem}
\usepackage{amsmath}
\usepackage{amsthm}
\usepackage[utf8]{inputenc}
\usepackage[english]{babel}
\usepackage{amssymb}
\usepackage{amsfonts}
\usepackage{tikz-cd}
\usepackage{enumerate}
\usepackage[utf8]{inputenc}
\usepackage{tikz}
\usepackage[top=1 in,bottom=1in, left=1 in, right=1 in]{geometry}

\usepackage{url}
\usepackage{hyperref}
\usepackage{mathrsfs}
\usepackage{lmodern}

\newcommand{\p}{\mathbb{P}}
\newcommand{\N}{\mathbb{N}}

\newcommand{\A}{\mathcal{A}}
\newcommand{\B}{\mathcal{B}}
\newcommand{\C}{\mathcal{C}}
\newcommand{\h}{\mathcal{H}(G,I)}
\newcommand{\s}{\mathcal{O}}

\newcommand{\Z}{\mathbb{Z}}

\renewcommand{\emptyset}{\varnothing}

\theoremstyle{plain}
\newtheorem{theorem}{Theorem}[section] 

\theoremstyle{definition}
\newtheorem{prop}[theorem]{Proposition}
\newtheorem{lemma}[theorem]{Lemma}

\newtheorem{Not}[theorem]{Notation}

\newtheorem{Remark}[theorem]{Remark}

\newtheorem{cor}[theorem]{Corollary}

\newtheorem{defn}[theorem]{Definition} 
\newtheorem{exmp}[theorem]{Example} 

\usepackage{mathtools}
\usepackage{multirow}

\DeclarePairedDelimiter\floor{\lfloor}{\rfloor}
\title{Iwahori Matsumoto presentation for modules of Iwahori fixed functions on symmetric spaces}
\author{Guy Shtotland }

\begin{document}

\maketitle

\setlength\parindent{0pt}

\begin{abstract}
     We study the space $S(X)^I$ of smooth functions on a symmetric space $X=H\backslash G$ invariant to the action of an Iwahori subgroup $I$, as a module over  $\mathcal{H}(G,I)$, the Iwahori Hecke algebra of a reductive $p$-adic group $G$. We present a description of this module that generalizes the description given to $\mathcal{H}(G,I)$ by Iwahori and Matsumoto.
\end{abstract}

\section{Introduction}
Let $F$ be a non archimedean local field with residue field $k$. We assume that the characteristic of $k$ is not 2.

Let $\mathbf{G}$ be a connected reductive group split over $F$ and let $G=\mathbf{G}(F)$. 
Let $\sigma:\mathbf{G}\rightarrow \mathbf{G}$ be an algebraic involution defined over $F$. We denote by $\mathbf{H}$ the algebraic group of fixed points of $\sigma$, and by $H$ the $F$-points of $\mathbf{H}$. We call
$X=H\backslash G$ a $p$-adic symmetric space. 

A fundamental problem in relative representation theory is the study of $H$-distinguished representations of $G$, that is, smooth representations $\pi$ of $G$ such that $Hom_H(\pi,\mathbb{C})\neq 0$. By Frobenius reciprocity the study of irreducible $H$-distinguished representations is the same (up to duality) as the study of quotients of the $G$ module $S(X)$, the space of locally constant compactly supported functions on $X$.

Let $I$ be an Iwahori subgroup of $G$ (defined in  \cite{PMIHES_1965__25__5_0}). The group $I$ splits the category $\mathcal{M}(G)$ of smooth representations of $G$ (see \cite{Benrstein_Center}). The subcategory of representations generated by their $I$ fixed vectors is the so-called principal block. The category of admissible representations generated by their $I$ fixed vectors was studied in \cite{Borel1976} and was shown to be equivalent to the category of finite dimensional $\mathcal{H}(G,I)$ modules, the equivalence is given by $V\rightarrow V^I$.
Here, $\mathcal{H}(G,I)=S(G)^{I\times I}$ is the Hecke Iwahori algebra of $G$ consisting of compactly supported locally constant $I$ bi-invariant functions on $G$.

A first step towards the study of smooth $H$-distinguished representations is the study of $\mathcal{H}(G,I)$ modules with a non-trivial equivariant map from $S(X)^I$. Thus, a description of $S(X)^I$  as a module over $\mathcal{H}(G,I)$ is a crucial step. In the present work, we do exactly this. Below we describe our main results in detail: Using the affine building $\B$ of $G$ we are able to describe the $I$ orbits on $X$, which give a linear basis of $S(X)^I$ over $\mathbb{C}$. Then we describe the action of the Iwahori Matsumoto generators of $\mathcal{H}(G,I)$ on a characteristic function of an $I$ orbit.

This paper is the first in a series of three papers about the module $S(X)^I$ over $H(G,I)$. In a future paper, we will describe the structure of the module $S(X)^I$ in the special case of $X$ being of minimal rank (see \cite{RESSAYRE20101784} for a definition). In an additional future paper we will use the results of this paper to answer the question of when the Steinberg representation of $G$ is distinguished with respect to $H$, under the assumption of $H$ being split.

\subsection{Main results}

 Let $W_{aff}$ be the extended affine Weyl group attached to the affine root system of $G$.

 We consider $X=H\backslash G$ with a right $G$ action. The Iwahori subgroup $I$ acts on $X$ from the right.

We construct an action of $W_{aff}$ on the set $X/ I$, which has finitely many orbits. We use this action to obtain a presentation of the module $S(X)^{I}$ over the Iwahori-Hecke algebra. This generalizes the well known presentation of the Iwahori-Hecke algebra given by Iwahori and Matsumoto. In the group case, where $G=H\times H$ and $H$ is embedded diagonally in $G$, our presentation of $S(H)^{I\times I}=\mathcal{H}(H,I)$ as a $\mathcal{H}(H,I)\times \mathcal{H}(H,I)$ module is exactly the presentation given in \cite{PMIHES_1965__25__5_0}.

Our first main result is the following, it is proven in Subsection \ref{the action section}.
\begin{theorem}\label{action} There is an explicit action of $W_{aff}$ on $X/I$. This action has finitely many orbits.
\end{theorem}

We now describe our results in more detail. 
A torus $T \subset G$ is called $\sigma$ stable if $\sigma(T)=T$. Let $\mathcal{T}_\sigma$ be the set of $\sigma$ stable maximal tori. For each $T\in \mathcal{T}_\sigma$ one can attach an extended affine Weyl group $W_{aff}(T)=N_{G}(T)/T^{0}$. Here, $T^0$ is the unique maximal compact subgroup of $T$. We also attach to $T$ a set $W^H_{aff}(T)$ of cosets of $W_{aff}(T)$ with respect to a subgroup $W_{H,aff}(T)=Im(N_H(T)\rightarrow N_G(T)/T^0)$  that is naturally embedded into $W_{aff}(T)$. 

Thus, 
$$W^H_{aff}(T)=W_{H,aff}(T)\backslash W_{aff}(T)$$

It is known that $H$ acts on $\mathcal{T}_\sigma$ with finitely many orbits (see \cite{Helminck}) and we let $A=H\backslash \mathcal{T}_\sigma$ be a set of representatives.

Theorem \ref{action} follows from the following description of the set $X/I$.
\label{1.2}
\begin{theorem}
 There is an explicit bijection between the $I$ orbits on $X$ and $\cup_{T\in A} W^H_{aff}(T)$
\end{theorem}

The proof is given in Section \ref{s3} in Theorem \ref{3.3}.

\begin{Remark}
    The same description was obtained independently by Paul Broussous, see \cite{paul2024orbitssymmetricsubgroupbruhattits}.
\end{Remark}

Using this bijection we define a right action of $W_{aff}$, the extended affine Weyl group of $G$, on the $I$ orbits on $X$. 

We denote this action by $x\times w$ for $w\in W_{aff}$ and $x\in X/I$.

We now turn our attention to the description of the action of the Iwahori Hecke algebra $\mathcal{H}(G,I)$ on $S(X)^I$.
We prove two results regarding this action: the first is a generalization of the classical Iwahori-Matsumoto relations, but it is not a complete description. The second result completes the description but is more involved.
We begin with the first result.

Let $\Tilde{\Delta}$ be the set of simple reflections inside $W_{aff}$ and let $\Omega$ be the group of length zero elements in $W_{aff}$ (see \ref{2.3} for the definition of $\Omega)$.

For any $w\in W_{aff}$ we denote $T_w=1_{IwI}\in \h$ the characteristic function of $IwI$.

The algebra $\mathcal{H}(G,I)$ is generated (as an algebra) by $\{T_s|s\in\Tilde{\Delta}\}$ and by $\{T_o|o\in\Omega\}$. Thus, in order to describe the action of $\mathcal{H}(G,I)$ it is enough to describe the actions of each $T_s$ for $s\in \Tilde{\Delta}$ and each $T_o$ for $o\in \Omega$.

In Section 5 (see Definition \ref{l_sigma}) we define a length function, $l_\sigma$, on the set of $I$ orbits on $X$. This length function is similar to the standard length function on $W_{aff}$. We prove the following  (see Theorem \ref{simpleformula}).

\begin{theorem}\label{1.3}
Let $x$ be an $I$ orbit on $X$, let $s\in \Tilde{\Delta}$ be a simple reflection, denote by $1_x$ the characteristic function of $x$.

\item if $l_\sigma(x\times s)>l_\sigma(x)$ then $1_xT_s=1_{x\times s}$
    \item if $l_\sigma(x\times s)<l_\sigma(x)$ then $1_xT_s=(q-1)1_x+q1_{x\times s}$
\end{theorem}

Unlike in the description of $\mathcal{H}(G,I)$ given by Iwahori and Matsumoto in \cite{PMIHES_1965__25__5_0}, $l_\sigma(x\times s)=l_\sigma(x)$ is a possibility and in this case the action of $T_s$ is more complicated and is not described by the Theorem above. Nevertheless, this result is sufficient to deduce that $S(X)^I$ is finitely generated over $\mathcal{H}(G,I)$.

For the second result we rely heavily on the geometry of the affine building of $G$.
We first prove that $1_x(T_s+1)$ is constant on its support and we present a description of this support in terms of chambers in the affine building of $G$.

It turns out that there are four options for the size of the support of $1_x(T_s+1)$. 
Moreover, we give a way to determine the value of $1_x(T_s+1)$ on its support by comparing the $l_\sigma$ lengths of the elements in the support.
To formulate the result we need further notations. 
\begin{Not}
Let $x$ be an $I$ orbit on $X$, $s\in \Tilde{\Delta}$ be a simple reflection,  and let $Z \subset H \backslash G/I$ be the support of $1_x(T_s+1)$. 

 Denote by $n_{max},n_{min}$ be the number of elements in $Z$ of maximal/minimal $l_\sigma$ length respectively.
 Meaning, let $l_{max}=\text{max}\{l_\sigma(z)|z\in Z\}$ and let $n_{max}=\#\{y\in Z|l_\sigma(y)=l_{max}\}$. The value $n_{min}$ is defined similarly.
 
 Let $\delta_{x,max}$ be 1 if $l_\sigma(x)=l_{max}$ and $0$ otherwise. The value $\delta_{x,min}$ is defined similarly. 
\end{Not}

\begin{theorem}\label{1.5}$1_x(T_s+1)$ is constant on its support, denote this constant by $\gamma$.
We have $n_{max},n_{min}\in\{1,2\}$.
If $l_\sigma(x)= l_\sigma(x\times s)$ then $$\gamma=\frac{q-1}{n_{max}}\delta_{x,max}+\frac{2}{n_{min}}\delta_{x,min}$$

\end{theorem}

This result is proven in Section \ref{s7}, see Theorem \ref{3.25}.

Our final result is the construction of a generic module $M_t$ over the generic Iwahori Hecke algebra $H_t$ which specializes to the module $S(X)^I$ over $H(G,I)$ at $t=1$ (see Theorem \ref{generic} for an exact formulation).

In the last section of this paper, we illustrate through the cases of $T\backslash SL_2$ and $Sp_{2n}\backslash SL_{2n}$ the concrete results one can obtain from our general theorems.

\subsection{Methods of proof}

We begin by noting that the study of $I$ orbits on $X$ can be translated into the study of $H$ orbits on $G/I$. One can identify $G/I$ with the set of chambers in the affine building of $G$ together with a $\Omega$ coloring. Thus $X/I$ can be replaced with $H$ orbits on the $\Omega$ colored chambers in the affine building $\B_G$.

Taking a cue from \cite{delorm}, the starting point of our work is the following:

\begin{prop}
Any chamber in $\B$ is contained in a $\sigma$ stable apartment. This apartment is {\bf unique} up to an action by an element of $H$. 
\end{prop}

This proposition is the key to the classification of orbits described in Theorem \ref{3.3}. The action of $W_{aff}$ on them in an immediate consequence.

To any element of  $x \in X/I$ we attach the characteristic function $1_{x}$. These functions form a basis for the vector space $S(X)^I$. Here $x$ is thought of as a $H$ orbit of $\Omega$ colored chambers in $\B_G$, the affine building of $G$.
A second tool we use, is our length function $$l_\sigma:  X/I \to \mathbb{N}$$ that is defined to be the distance between a representative $\C_x$ of the orbit $x$, and $\sigma(\C_x)$ in the building $\B_G$. 

Theorem \ref{1.3} describes the action $T_{s}$ on $1_{x}$, where $s$ is a simple reflection, in the case where $l_\sigma(x\times s)\neq l_\sigma(x)$. 

The proof of Theorem \ref{1.3} relies on the following Lemma (see Lemma \ref{3.9} in the text):

\begin{lemma}
Let $x\in H\backslash G/I$ and let $s\in \Tilde{\Delta}$ be a simple reflection.
Let $\C$ be the chamber fixed by $I$. Choose $g\in G$ a representative of the orbit $x$ and choose a $\sigma$-stable apartment $\A$ containing both  $g\C,gs\C$.
Finally let $f=g\C\cap gs\C$ and let $\Pi$ be the intersection of all $\sigma$ stable apartments that contain $f$.

Then we have the following criteria:
\begin{itemize}

\item $l_\sigma(x \times s)>l_\sigma(x)$ if and only if $g\C\subset \Pi$.
    \item $l_\sigma(x\times s)<l_\sigma(x)$ if and only if $gs\C\subset \Pi$.
\end{itemize}
\end{lemma}

The case $l_\sigma(x\times s)=l_\sigma(x)$ is more complicated and requires a different approach.
We introduce a group $G_{f}$ obtained as the reductive quotient associated to $f\cap \sigma(f)$ (see subsection \ref{s2.3} for properties of the reductive quotient). The group $G_{f}$ is a reductive group of semisimple rank one, we make the observation that the $H$ orbits appearing in the support of $1_x(T_s+1)$ are in bijection with the orbits of a certain subgroup $\Tilde{H}$ of $G_f$ on the flag variety of $G_{f}$ (See Proposition \ref{orbit correspondacne}). All this is relevant to the calculation of $1_{x}T_{s}$ because a closer look at the key calculation (see Proposition \ref{key_computation}): $$1_{x}(T_{s}+1)=\gamma_{f,g}D_{f}$$
show that $\gamma=\gamma_{f,g}$ is the size of the orbit of $P_{f} \cap H$ on $g\C$ (see Proposition \ref{6.1}), Here $P_f$ is the point-wise stabilizer of $f$. Using the classification of reductive groups of semisimple rank 1 we are able to describe the possible values of $\gamma$.

\subsection{Structure of the paper}

In Section 2 we provide some necessary background on the affine Weyl group of $G$, the affine building of $G$, parahoric group schemes, the Iwahori Hecke algebra $\mathcal{H}(G,I)$ and symmetric spaces.

In Section 3 we prove Theorem \ref{3.3} describing the set $X/I$ of $I$ orbits on $X$ in terms of quotients of the extended affine Weyl group. We then introduce an action of $W_{aff}$ on $X/I$ and verify Theorem \ref{action}.

In Section 4 we perform a key computation describing the action of a Iwahori-Matsumoto generator $T_{s}$ of $\mathcal{H}(G,I)$ on a characteristic function $1_{x}$ for $x \in X/I$. The answer we obtain is in terms of a combinatorial quantity $\gamma$ which is the size of a specific $H$-orbit of chambers in the affine building $\B_G$ attached to $x$ and to $x \times s$.

In Section 5 we prove Theorem \ref{simpleformula} that provides Iwahori-Matsumoto type formulas for the action of $\mathcal{H}(G,I)$ on $S(X)^I$. Although our formulas do not provide a full description of the module $S(X)^I$, they are sufficient to deduce that $S(X)^I$ is a finitely generated module over $\mathcal{H}(G,I)$. 

In Section 6 we discuss actions of involution on parahoric group schemes and obtain a bijection between $H$ orbits on certain chambers in $\B_G$ and the orbits of a symmetric subgroup of $G_f$, a reductive quotient, on the flag varieties of $G_f$. 

In Section 7 we use the results of Section 6 to conclude the analysis of the action of an Iwahori-Matsumoto generator of $\mathcal{H}(G,I)$ on a characteristic function of an orbit.

In Section 8 we use the results of the previous sections to introduce a generic module $M_{t}$ over the generic Iwahori-Hecke algebra that is described using generators and relations.

In Section 9 we provide two examples for the results of the previous sections.

\subsection{Related works}
Regarding Theorem \ref{1.2}, we mention that the generalizations of the Cartan decomposition 
to the context of symmetric and spherical spaces was extensively studied. We mention here the work \cite{Off04} of Offen in some special cases as well as the work by \cite{Mao-Rallis09}. In the symmetric spaces case by Delorme and Secherre
 \cite{delorm} and Benoist and Oh \cite{polar_decomposition}.  We also mention the work \cite{Sak12} of Sakellaridis considering split groups on a wide class of spherical varieties.

 In these works the authors study the orbits of $K_0$, a maximal compact subgroup of $G$, on $X$. 
Particularly in \cite{delorm}, Delorme and Secherre, used the affine building of $G$ to give a description of the $K_0$ orbits on $X$. Our technique is inspired by their work to give a description of the $I$ orbits on $X$. 

In \cite{RichardsonSpringer1990} an action of the Weyl group on the Borel orbits of a symmetric space is introduced. In \cite{knop} this constructions was extended to the action of the Weyl group on the Borel orbits of any spherical variety.

The action of $W_{aff}$ we introduce on $X/I$ is an affine version of this in the case of symmetric spaces. Applying our arguments to the spherical building (instead of the affine building) recovers these results in the case of symmetric spaces.

The module $S(X)^{K_0}$ over the Hecke algebra $H(G,K_0)$ was studied extensively, for example by Hironaka \cite{10.2969/jmsj/05130553} for certain spherical and symmetric varieties and by Sakellaridis \cite{sakellaridis2013spherical} for a wide class of spherical ones. After determining the spherical functions, they used a Satake type transform in order to obtain a description of the module. Our study of $S(X)^I$ for $X$ symmetric is based on geometry and seems to be new. 
Recently the modules of $I$ fixed vectors of the Gelfand Graev models and other models were considered. See in particular the results of \cite{ChaSav18}.

A study of Hecke modules in the algebraic case was initiated in the work of Lusztig-Vogan \cite{Lusztig1983} where they study the Grothendieck group of the category of $l$-adic constructible $H$ equivariant  sheaves on the flag variety $G/B$ as a module over $H(G,B)$ where $B$ is a Borel subgroup of $G$. A version of this for the Iwahori case was obtained in \cite{chen2025singularitiesorbitclosuresloop}. Although our methods of proof are different, our formulas resemble their description.

\subsection{Acknowledgement}
Most of the results presented in this work were obtained as part of my MSc studies at Ben-Gurion University of the Negev. I would like to thank my MSc advisor (currently my PhD advisor), Eitan Sayag, for teaching me a lot of mathematics and also for many helpful discussions about my research and his many helpful comments during the writing of this text. I would also like to thank Avraham Aizenbud, Joseph Bernstein, Daniel Disegni and Nadya Gurevich. Special thanks to Paul Broussous for pointing out a mistake in a previous version of the proof of Proposition \ref{main lemma}. I would like to thank the referees for their careful reading and their helpful comments and suggestions. I was partially supported by ISF grant No. 1781/23 during the work on this paper.

\section{Notations and Preliminaries}

In this section we present several well known results that we use in this work. These results regard the extended affine Weyl group, $W_{aff}$ of $G$, the affine building $\B_G$, parahoric group schemes, the Iwahori Hecke algebra $\mathcal{H}(G,I)$ and $p$-adic symmetric spaces.

\subsection{The extended affine Weyl group of $G$}
\label{2.1}

\begin{defn}[Coxeter group]
A Coxeter group $W$ is a group generated by reflections $W=<s_1,...,s_n>$, subject to the relations $(s_is_j)^{m_{ij}}=1$, $m_{ij}\in \mathbb{N}\cup\{\infty\}$ and $m_{ii}=1$ for each $i$.
    
    There is a length function $l$ on $W$, $l(w)$ is defined to be the minimal number $k=l(w)$ such that there are $i_1,...,i_k\in \{1,...,n\}$ for which $w=s_{i_1}\cdot...\cdot s_{i_k}$.

\end{defn}
Let $T$ be a maximal split torus of $G$, we denote by $T^0$ the unique maximal compact subgroup of $T$, if $T\cong (F^\times)^n$ then $T^0\cong (\s^\times)^n$, here $\s$ is the integer ring of $F$.

We denote by $W(T)=N_G(T)/T$ the Weyl group of $G$ and $T$, or just by $W$ if it is clear from the context with respect to which maximal torus the Weyl group is taken. Similarly, we denote by $W_{aff}(T)=N_G(T)/T^0$ or just by $W_{aff}$ the extended affine Weyl group of $G$ and $T$.

Let $(X^*(T),\Phi,X_*(T),\Phi^\vee)$ be the root datum (see \cite{Springer1998} for the definition) that corresponds to $T$.

We will use the notation $X_*=X_*(T)$ and $X^*=X^*(T)$. 

We choose a subset of positive roots $\Phi^+\subset \Phi$ and simple roots $\Delta\subset \Phi^+$.

\begin{theorem}
The group $W$ is isomorphic to the Coxeter group generated by reflections over $\Delta$ (see \cite{LAG}).
\end{theorem} 

For any root $\alpha\in\Phi$ denote by $\alpha^\vee\in\Phi^\vee$ the coroot of $\alpha$ (see \cite{Springer1998} for a definition).

\begin{prop}\label{2.5}
There is a perfect pairing $X_*\times X^*\rightarrow Aut(\mathbb{G}_m)=\Z$ given by composition. 
\end{prop}

Let $V=X_*\otimes\mathbb{R}$, for any $\alpha\in \Phi$ let $s_\alpha\in N_G(T)$ be an element that induces the reflection over $\alpha^\vee$ in $V$ (see \cite{Landvogt1995ACO} Lemma 0.19).
We consider $\alpha\in \Phi$ as an element of $V^*$. For any integer $k\in\Z$ we consider the affine root $\alpha+k$ and the following affine reflection on $V$ defined by $x\mapsto s_\alpha(x)-k\alpha^\vee$.

By Proposition \ref{2.5} each root $\alpha$ induces a function on $V$. Each affine root $\alpha+k$ also induces a function on $V$.

Denote by $H_{\alpha+k}$ the wall $H_{\alpha+k}=\{x\in V|\alpha(x)+k=0\}\subset V$.

The closures of the connected components of $V-\cup_{\alpha\in \Phi,k\in\Z}H_{\alpha+k}$ will be referred to as the chambers of $V$. 

\begin{Remark}
    Our definition of chambers is slightly different than the standard one. In the standard definition, the chambers are open subsets of $V$.
\end{Remark}

Let $\C$ be the unique chamber on which all the simple roots are positive and which contains $0$. $\C$ is called the fundamental chamber.

We define a length function $l:W_{aff}(T)\rightarrow\mathbb{N}$. For $w\in W_{aff}(T)$ define $l(w)$ to be the number of walls between $\C$ and $w\C$.

The group $W$ is generated by $\{s_\alpha|\alpha\in \Delta\}$ (see \cite{Springer1998} Theorem 7.1.9) and it acts on $V$ preserving the lattice $\Z\Phi^\vee$. Let $W_a=\Z\Phi^\vee\rtimes W$, It is a subgroup of affine transformations of $V$.

\begin{theorem}\label{2.3}
The group $W_a$ is a Coxeter group generated by reflections over the facets of $\C$. The function $l$ restricted to $W_a$ is the equal to the Coxeter length function on $W_a$. Additionally $W_{aff}(T)=W_{a}\rtimes \Omega$, where $\Omega=l^{-1}(0)\cong X^\vee/\Z\Phi^\vee$.
\end{theorem}

Let $\Tilde{\Delta}$ be a set of elements in $N_G(T)$ which induce on $V$ the reflections over the facets of the fundamental chamber $\C$.

Denote by $W_{aff}$ the abstract group $W_{a}\rtimes \Omega$ to which all the $W_{aff}(T)$ are isomorphic. Here $W_{a}$ is considered as an abstract Coxeter group and $\Omega$ as an abstract abelian group.

\subsection{The affine building of $G$}
\label{2.2}
There are several ways to construct the affine building of a split $p$-adic group $G$, we follow the construction presented in \cite{delorm} (for more details see \cite{Landvogt1995ACO}).

\begin{prop}\label{2.4}
Let $\mathbf{T}\subset \mathbf{G}$ be a maximal torus. Let $\Phi$ be its root system.
For a root $\alpha\in \Phi$ there is a unipotent group $\mathbf{U}_\alpha\subset G$ and an isomorphism $u_\alpha:\mathbb{G}_a\rightarrow \mathbf{U}_\alpha$ such that for any $t\in T$, $t^{-1}u_\alpha(a)t=u_\alpha(\alpha(t)a)$ (see \cite{LAG} Proposition 21.9).
\end{prop}

Let $U_\alpha=\mathbf{U}_\alpha(F)$.

There is a filtration on $\mathbb{G}_a(F)$ given by powers of the uniformizer times $\s$, this induces via $u_\alpha$ a filtration on $U_\alpha$. Let $\pi$ be a uniformizer and denote $U_{\alpha,k}=u_\alpha(\pi^k\s)$ for $k\in \mathbb{Z}$. This yields a filtration on $U_\alpha$.

The group $N_G(T)$ acts on $X_*(T)$ and we can extend this action to an affine action on $V$. As $T^0$ acts trivially we obtain an action of $W_{aff}(T)$ on $V$.

\begin{Not}\label{2.7}
    Let $L\subset V$ be any subset.
    
    We define $N^1_L$ to be the group of elements in $N_G(T)$ that fixes $L$ pointwise.
    
    We also denote $U_L=\Big \langle U_{\alpha,k}\Big|(\alpha+k)|_L\geq 0\Big \rangle$.
    
    We denote by $P^1_L$ the group generated by $U_L$ and $N^1_L$. 
    
    For a single point $v\in V$ we denote $U_v=U_{\{v\}}$ and $P_v=P_{\{v\}}$.
    
\end{Not}

Now we can give a description of the enlarged (also called extended) affine building $\B_G$ of the $p$-adic group $G$.

\begin{defn}
Let $G$ be a reductive connected $p$-adic group. Let $T$ be a maximal split torus of $G$ and let $V=X_*(T)\otimes \mathbb{R}$.

As a set $\B^{T}_G=G\times V/\sim$ where $\sim$ is the equivalence relation $(g,x)\sim (h,y)$ if there is $n\in N_G(T)$ such that $nx=y$ and $g^{-1}hn\in U_x$. $\B^{T}_G$ is called the enlarged affine building of $G$ associated to $T$.

Let $Z=Z(G)$ be the center of $G$, let $\mathcal{Z}=X_*(Z)\otimes \mathbb{R}$. $\mathcal{Z}$ is a subset of $\B^{T}_G$ and we have $\B^{T}_G=\mathcal{Z}\times \B^{nen,T}_G$. $\B^{nen,T}_G$ is called the non-enlarged affine building of $G$.  
\end{defn}

\begin{theorem}
    All maximal split tori of $G$ are conjugate (see Theorem 18.14 in \cite{Milne_2017}).
\end{theorem}

The next theorem says that $\B^T_G$ is independent of $T$.

\begin{theorem}\label{2.9}
Let $T_1,T_2\subset G$ be two maximal split tori, and let $g\in G$ such that

$g^{-1}T_1g=T_2$. Conjugation by $g$ induces an isomorphism of $G$ topological spaces between  $\B^{T_1}_G$ and $\B^{T_2}_G$. (See \cite{Landvogt1995ACO} Proposition 9.20).
\end{theorem}

We denote by $\B_{G}$ the enlarged affine building of $G$ constructed from some maximal split torus, and by $\B^{nen}_{G}$ the non-enlarged affine building of $G$.

We identify $V$ with the subset $\{(1,v)|v\in V\}\subset G\times V/\sim$.

\textbf{Chambers and apartments}

\begin{Not}\label{apartments}
The apartments of the enlarged building are defined to be the sets of the form $gV=\{(g,v)|v\in V\}/\sim$ which are subsets of $\B_G$. The apartments of the non-enlarged building are the quotient of the apartments of the enlarged building with respect to $\mathcal{Z}$.

We denote apartments by curly $\A$.  There is a correspondence between maximal split tori and apartments, the apartment $gV$ corresponds to the torus $gTg^{-1}$.

All apartments have a structure of an affine space.

The chambers of the enlarged building are defined to be the sets of the form $g\C=\{(g,v)|v\in \C\}/\sim$. Chambers of $\B^{nen}_G$ are quotients of chambers of $\B_G$ with respect to $\mathcal{Z}$.
\end{Not}

Let $n$ be the split rank of $G$.

\begin{theorem}\label{2.10}
$\B_G$ is a poly-simplicial complex that satisfies the following properties:
\begin{enumerate}
    \item Every $n-1$ poly-simplex in an apartment $\A$ lies in exactly two adjacent $n$ poly-simplices of $\A$ and the graph of adjacent $n$ poly-simplices is connected.
    \item Any two poly-simplices in $\B_G$ lie in some common apartment $\A$.
    \item If two poly-simplices both lie in apartments $\A$ and $\A'$, then there is a simplicial isomorphism of $\A$ onto $\A'$ fixing the poly-simplices. This isomorphism can be taken to be an action by an element of $G$.
\end{enumerate} (For more details see Section 9 of \cite{Landvogt1995ACO}).

\end{theorem}

The poly-simplices of dimension $n$ in $\B_G$ are the chambers of $\B_G$. 

The group of elements in $G$ that fix a specific chamber in $\B_G$ pointwise is called an Iwahori subgroup of $G$ (see Definition 4.1.3 in \cite{BruhatTitsTA}).

The next theorem summarizes some important properties of the building mentioned in subsections 1.11 and 1.13 in \cite{delorm}.

\begin{theorem}\label{2.11}

\begin{enumerate}
    \item $\B_G$ is the union of its apartments and $G$ acts transitively on them.
    \item $N_G(T)$ is the stabilizer of $V$ in $\B_G$. 
    \item The group $U_{\alpha,k}$ fixes pointwise the points in $V$ on which the affine root $\alpha+k$ is non-negative.  
    \item For any $L\subset \A$, the group $P^1_L$ is exactly the group of elements that fix $L$ pointwise.
    \item For any $g\in G$ there exists $n\in N_G(T)$ such that $gx=nx$ for any $x\in V\cap g^{-1}V$.
\end{enumerate}
\end{theorem}

\begin{theorem}[The Affine Bruhat Decomposition, \cite{PMIHES_1965__25__5_0}]\label{affine bruhat decomposition}
Let $T$ be a maximal split torus of $G$, let $\A$ be the corresponding apartment in $\B_G$, let $\C\subset\B_G$ be a chamber and let $I$ be the Iwahori subgroup that corresponds to $\C$. Then $G=IW_{aff}(T)I$.
\end{theorem}

\begin{prop}\label{2.12}
Let $\C$ be a chamber of $\B_G$, let $I$ be the Iwahori subgroup which fixes $\C$ pointwise and let $J$ be the group that fixes $\C$ as a set (not necessarily pointwise).

    Using Theorems \ref{affine bruhat decomposition} and \ref{2.3} we can construct a group homomorphism $\omega:G\rightarrow\Omega$ by $\omega(i_1xoi_2)=o$ for $i_1,i_2\in I$, $x\in W_a$ and $o\in \Omega$.
    \begin{enumerate}
        \item The homomorphism $\omega$ is well defined.
        \item $I$ is normal in $J$ and the homomorphism $\omega$ induces an isomorphism $J/I\cong \Omega$.
    \end{enumerate}
\end{prop}

\begin{defn}
    For any subset $L$ inside an apartment $\A$, we define $P_L=P^1_L\cap Ker(\omega)$. This is the group of elements in $Ker(\omega)$ that fix $L$ pointwise.
\end{defn}

Let $I$ be an Iwahori subgroup of $G$. Proposition \ref{2.12} allows us to think about $G/I$ as the set of $\Omega$ colored chambers in the building of $G$. Denote by $C_G$ the set of chambers in $\B_G$, we denote the set of $\Omega$ colored chambers by $C_\Omega(G)=C_G\times \Omega$. The group $G$ acts on $C_\Omega(G)$ diagonally, when the action on $\Omega$ is given by $\omega:G\rightarrow \Omega$, meaning $g(\C,o)=(g\C,\omega(g)o)$.

\textbf{Convexity of apartments in the building}

The convexity of apartments has two forms, the first is given by the following theorem (See \cite{Landvogt1995ACO} Proposition 9.6).

\begin{theorem}\label{2.13}
The intersection of any two apartments $\A$ and $\A'$ is convex with respect to the affine geometry on $\A$. Meaning that for any $x,y\in \A\cap \A'$ the segment $\{ax+(1-a)y|a\in [0,1]\}$ inside the affine space $\A$ is contained in $\A\cap \A'$. 
\end{theorem}

For the second form of convexity we need another definition.

\begin{defn}\label{2.14}

    Let $\C_1$ and $\C_2$ be two chambers, a path of adjacent chambers which starts at $\C_1$ and ends at $\C_2$ is called a gallery. It will be referred to as a minimal path or a minimal gallery if the number of chambers in it is equal to the minimal number needed to pass from $\C_1$ to $\C_2$ via adjacent chambers.
\end{defn}
    
    \begin{theorem}\label{2.15}
     If $\C_1,\C_2$ are chambers in an apartment $\A$ then any minimal gallery between them is contained in $\A$. (See \cite{Brown1989} Chapter 4 Section 4).
    \end{theorem}

\begin{theorem}[Iwahori decomposition]\label{2.16}
Let $L\subset \A$ be a set with non empty interior (with respect to the affine topology on $\A$). For any root $\alpha$ let $\alpha_L$ be the affine root of the form $\alpha+k$ with the minimal $k$ that is non negative on $L$. Let $U_L^+$ be the group generated by $U_{\alpha_L}$ for $\alpha\in\Phi^+$ and $U_L^-$ be the group generated by $U_{\alpha_L}$ for $\alpha\in\Phi^-$, then we have
\begin{equation}
    P_L=T^0U_L=T^0U_L^+U_L^-
\end{equation}
  Furthermore we have  $U_L^+=\Pi_{\alpha\in\Phi^+}U_{\alpha_L}$ when the product can be taken in any order.

(See \cite{BruhatTitsTA} Proposition 13.2.5).
\end{theorem}

\begin{Remark}
This is a generalization of the classical Iwahori decomposition. In the case of $G=GL_n(F)$ the statement of the classical Iwahori decomposition is as follows:

Let $K=GL_n(\s)$, $I\subset K$ be the subgroup of matrices reduced to upper-triangular when taken mod $\pi$. Let $B\subset GL_n(F)$ be the Borel subgroup of upper-triangular matrices and $B^t$ the Borel subgroup of lower triangular matrices.  Let $T$ be the maximal split torus of diagonal matrices. Let $U$ be the unipotent of $B$ and $U^t$ the unipotent of $B^t$. Denote $I^+=I\cap U$ and $I^-=I\cap U^t$. The classical Iwahori decompisiton states that 
\begin{equation*}
    I=I^+I^-T^0
\end{equation*}

This can be recovered from Theorem \ref{2.16} by taking $L=\C$ the fundamental chamber. Then we will have $P_L=I$, $U^+_L=I^+$ and $U^-_L=I^-$.
\end{Remark}

\textbf{The building and automorphisms}

\begin{theorem}[The building construction is a map of groupoids, \cite{PMIHES_1984__60__5_0} §4.2.12]\label{2.17}
 Any algebraic automorphism $a$ of $G$ induces an automorphism $a^*$ of $\B_G$ which is an affine map on each apartment and satisfies $a^*(gx)=a(g)a^*(x)$ for $g\in G$ and $x\in \B_G$.

\end{theorem}

\subsection{Parahoric group schemes}\label{s2.3}
    In this subsection we describe certain group schemes defined over $\s$, related to the Burhat Tits building, we follow \cite{Tits1979ReductiveGO}, \cite{McNinch2010LeviDO}, \cite{Landvogt1995ACO} and \cite{BruhatTitsTA}. In this subsection we will refer to subsets of $\B_G$ which are bounded mod center, meaning their image in $\B_G^{nen}$ is bounded, as bounded subsets.

    Let $\s$ be the integer ring of $F$.

    \begin{theorem}
    Let $L$ be a bounded set which is contained in some apartment of $\B^F_G$.
        There exists a unique smooth affine group scheme ${\mathcal G}_L$, defined over $\s$ such that:
        \begin{enumerate}
            \item The generic fiber of ${\mathcal G}_L$ is $G$.
            \item We have ${\mathcal G}_L(\s)\cong P_L$.
        \end{enumerate}

        (See Section 8.3 in \cite{BruhatTitsTA}. The notations in \cite{BruhatTitsTA} are $G(k)^0_\Omega$ for the groups $P_L$, and $\mathscr{G}^0_\Omega$ for ${\mathcal G}_L$).
    \end{theorem}

In our case, where $G$ is split over $F$, the group ${\mathcal G}_L$ can be constructed explicitly, as done in Subsection 6.2 of \cite{McNinch2010LeviDO}.

Let $\mathbf{T}$ be the group scheme of a maximal split torus of $G$ and let $T=\mathbf{T}(F)$ be a torus inside $G$. Let $\A$ be an apartment of $\B_G$ that corresponds to $T$. Let $L$ be a bounded subset of $\A$.

\begin{defn}
    Let $\alpha+k$ be an affine root of $T$ where $\alpha$ is a root of $T$ and $k\in\Z$. Let $u_\alpha:\mathbb{G}_a\rightarrow \mathbf{U}_\alpha$ be the isomorphism from Proposition \ref{2.4}. Let $\pi$ be a uniformizer of $\s$ and let $\mathbb{G}^k_a$ be the smooth $\s$ scheme whose $\s$ points are $\pi^k\s$. Let ${\mathcal U}_{\alpha+k}$ be the smooth group scheme obtained from $\mathbb{G}^k_a$ and $u_\alpha$ such that its $\s$ points are $u_\alpha(\pi^k\s)$.
\end{defn}

\begin{theorem}\label{open subscheme}

    Let $L$ be a bounded subset of $\A$. The inclusion of $T^0$ and $U_{\alpha_L}$ into $P_L$ can be extended to isomorphisms from $\mathbf{T}$ and ${\mathcal U}_{\alpha_L}$ onto closed sub-schemes of ${\mathcal G}_L$. 

    Like in Theorem \ref{2.16} we can define  ${\mathcal U}^+_{L}=\Pi_{\alpha\in \Phi^+}{\mathcal U}_{\alpha_L}$ and similarly ${\mathcal U}^-_{L}$. The product of the maps from ${\mathcal U}_{\alpha_L}$ to ${\mathcal G}_L$ yield an isomorphism from ${\mathcal U}^-_{L}$ and ${\mathcal U}^+_{L}$ onto closed sub-schemes of ${\mathcal G}_L$.

    Lastly, The product ${\mathcal U}^+_{L}\times \mathbf{T}\times {\mathcal U}^-_{L}$ maps isomorphiclly onto an open sub-scheme of ${\mathcal G}_L$.

    For the proof see either 6.1 in \cite{Landvogt1995ACO} or 8.3.14 in \cite{BruhatTitsTA}.
\end{theorem}

    Denote by $\overline{{\mathcal G}}_{L}$ the special fiber of ${\mathcal G}_{L}$. $\overline{{\mathcal G}}_{L}$ is an algebraic group over the residue field $k$.

    \begin{prop}\label{root system}
    Let $\Phi_L$ be the set of roots constant on $L$. The root system of  $\overline{{\mathcal G}}_{L}/R_u(\overline{{\mathcal G}}_{L})$ is $\Phi_L$. (See 8.4.10 in \cite{BruhatTitsTA})
    \end{prop}

    The following is stated without a proof in \cite{Tits1979ReductiveGO} and is Theorem A of \cite{McNinch2010LeviDO}.

\begin{theorem}\label{Levi decomposition}
    $\overline{{\mathcal G}}_{L}$ has a unique Levi factor which contains $\mathbf{T}(k)$. Here a Levi factor means a direct complement to the unipotent radical.
\end{theorem}

\begin{prop}\label{restriction maps}
    Let $L_1\subset L_2$ be two bounded subsets of some apartment, the identity map $G\rightarrow G$ extends in a unique way to a $\s$ group homomorphism $res^{L_2}_{L_1}: {\mathcal G}_{L_2}\rightarrow {\mathcal G}_{L_1}$. (See either 6.2 in \cite{Landvogt1995ACO} or 8.3.17 in \cite{BruhatTitsTA}).
\end{prop}

\begin{prop}\label{face isomorphism}
    Let $L$ be any bounded subset of an apartment and let $f$ be a facet of some chamber such that $f\subset L$ and $f$ is maximal with this property. Assume that $f\subset L$, then $res^{L}_{f}$ is an isomorphism on the reductive quotients of the special fibers. (See 3.5 in \cite{Tits1979ReductiveGO}).
\end{prop}

        Let $L_1\subset L_2$ be two facets of some chamber in the building, the map $res^{L_2}_{L_1}$ induces a map $\overline{res}^{L_2}_{L_1}$ from $\overline{{\mathcal G}}_{L_2}$ to $\overline{{\mathcal G}}_{L_1}/R_u(\overline{{\mathcal G}}_{L_1})$.

    \begin{prop}\label{affine to spherical building}
    Let $L$ be a facet of some chamber in the building, denote by $X_L$ the poset of all facets of any chamber which contain $L$, the ordering is with respect to inclusion.  The map $$L'\rightarrow image(\overline{res}^{L'}_{L})$$ is an order preserving bijection between $X_L$ and the spherical building of $\overline{{\mathcal G}}_{L}/R_u(\overline{{\mathcal G}}_{L})$ (See 9.22 in \cite{Landvogt1995ACO}).
    \end{prop}

\subsection{The Iwahori Hecke algebra of $G$}

In this subsection we summarize some results from \cite{PMIHES_1965__25__5_0}.

Fix a maximal split torus $T$ of $G$ and an Iwahori subgroup $I$ of $G$ such that $T^0\subset I$.
\begin{defn}
The Iwahori Hecke algebra of $G$ is the algebra of $I$ bi-invariant locally constant compactly supported functions on $G$, it is denoted by $\mathcal{H}(G,I)$.

The multiplication operation is given by convolution with respect to a Haar measure $m$ on $G$, $f*g(y)=\int f(yx^{-1})g(x)dm(x)$.
\end{defn}

For any $g\in G$ we denote by $T_g$ the function $T_g=1_{IgI}$.

The following is the well known Iwahori-Matsumoto description of the Iwahori Hecke algebra.

Let $q$ be the size of the residue field of $F$.

\begin{theorem}\label{2.20}
    $\mathcal{H}(G,I)$ is spanned over $\mathbb{C}$ by $T_w=1_{IwI}$ for $w\in W_{aff}(T)$. Let $s\in \Tilde{\Delta}$ be a simple reflection and let $o\in \Omega$ be an element of length 0.
    \begin{enumerate}
        \item if $l(sw)>l(w)$ then $T_sT_w=T_{sw}$.
        \item if $l(sw)<l(w)$ then $T_sT_w=(q-1)T_w+qT_{sw}$.
        \item $T_{o}T_w=T_{ow}$ for any $w\in W_{aff}$.
    \end{enumerate}
    \end{theorem}
    \begin{cor}\label{2.21}
    
    $\mathcal{H}(G,I)$ is generated as an algebra by $T_s=1_{IsI}$ for $s\in \Tilde{\Delta}$ and by $T_o$ for $o\in \Omega$. Moreover, the relations on these generators are generated by:
        \begin{enumerate}
            \item For $s\in \Tilde{\Delta}$, $(T_s+1)(T_s-q)=0$.
            \item if $s_1,s_2\in\Tilde{\Delta}$ such that $s_1\neq s_2$ and $(s_1s_2)^m=1$ then $T_{s_1}T_{s_2}T_{s_1}...=T_{s_2}T_{s_1}T_{s_2}...$  where each product contains $m$ elements.
            \item For $o_1,o_2\in \Omega$, $T_{o_1}T_{o_2}=T_{o_1o_2}$.
            \item For $o\in \Omega$ and $s\in\Tilde{\Delta}$ we have $oso^{-1}\in\Tilde{\Delta}$ and $T_oT_s=T_{oso^{-1}}T_o (=T_{os})$.
         \end{enumerate}
    \end{cor}

\subsection{Involutions and symmetric subgroups}

Let $\sigma:\mathbf{G}\rightarrow\mathbf{G}$ be an algebraic involution, we denote by $\mathbf{H}$ the group of $\sigma$ fixed elements and by $H$ the $F$ points of $\mathbf{H}$.

We call $X=H\backslash G$ a $p$-adic symmetric space. A $p$-adic symmetric space is an $l$ space (See 1.1 in \cite{Bernshtein1976REPRESENTATIONSOT} for a definition).

Tori $T$ in $G$ such that $\sigma(T)=T$ are called $\sigma$ stable.

An apartment $\A$ is called $\sigma$ stable if $\sigma(\A)=\A$, notice that $\sigma$ stable maximal tori correspond to $\sigma$ stable apartments under the correspondence mentioned in \ref{apartments}.

By Theorem \ref{2.17}, $\sigma$ defines an involution on $\B_G$ which we also denote by $\sigma$.

Proposition 2.1 of \cite{delorm} is the following:

\begin{theorem}\label{2.22}
Let $L\subset\B_G$ be a $\sigma$ stable set with a non empty interior which is contained in some apartment. Then $L$ is contained in a $\sigma$ stable apartment.
\end{theorem}

\begin{cor}\label{2.23}
Any chamber is contained in a $\sigma$ stable apartment.
\end{cor}

\begin{proof}
Let $\C$ be a chamber, $\sigma(\C)$ is another chamber, let $L=\C\cup\sigma(\C)$. By Theorem \ref{2.10}, $L$ is contained in some apartment. The set $L$ has a non empty interior as $\C$ has a non empty interior and $L$ is $\sigma$ stable.

By Theorem \ref{2.22}, $L$ is contained in a $\sigma$ stable apartment. 
\end{proof}

Consider the action of $H$ on the building $\B_G$ of $G$.
Translating to the language of the building we can deduce from Proposition 6.10 and Corollary 6.16 of \cite{Helminck} the following:

\begin{theorem}\label{Fin_Sig_Apart}
 $H$ acts on the set of $\sigma$ stable apartments with finitely many orbits.
\end{theorem}

\begin{defn}\label{convolution action}
    Let $S(X)^I$ be the space of $I$ invariant locally constant compactly supported functions on $X$. The algebra $H(G,I)$ acts from the right on $S(X)^I$. The action is by convolution with respect to a Haar measure $m$.

    For $f\in S(X)^I$ and $h\in H(G,I)$ we define $f*h(x)=\int f(xg^{-1})h(g)dm(g)$.
\end{defn}

\begin{prop}
    Definition \ref{convolution action} gives a well defined right action of $H(G,I)$ on $S(X)^I$. 
\end{prop}

\begin{proof}
    We need to check that for $f\in S(X)^I$ and $h_1,h_2\in H(G,I)$ we have $(f*h_1)*h_2=f*(h_1*h_2)$.

    We have $f*h_1(x)=\int f(xg_1^{-1})h_1(g_1)dm(g_1)$ and so 
    
    $(f*h_1(x))*h_2=\int \int f(xg^{-1}_2g_1^{-1})h_1(g_1)dm(g_1)h_2(g_2)dm(g_2)$.

    Denote $g_3=g_1g_2$, we have $(f*h_1)*h_2(x)=\int \int f(xg^{-1}_3)h_1(g_3g_2^{-1})dm(g_3)h_2(g_2)dm(g_2)=\int f(xg^{-1}_3) \int h_1(g_3g_2^{-1})h_2(g_2)dm(g_2) dm(g_3)= \int f(xg^{-1}_3) (h_1*h_2)(g_3)dm(g_3)=f*(h_1*h_2)(x)  .$
\end{proof}

\section{Iwahori orbits on symmetric spaces}\label{s3}

In this section, we give a description of the set of $I$ orbits on $X$ and define an action of $W_{aff}$ on them.

By Theorem \ref{2.12}, the set $G/I$ is in bijection with the set of $\Omega$ colored chambers in the affine building of $G$. The orbits of $I$ on $X$ are the same as the orbits of $H$ on $G/I$ or on the set of $\Omega$ colored chambers in $\B_G$.

In order to describe the $H$ orbits on $G/I$ we will need the following proposition about the $\sigma$ stable apartments.

\begin{prop}\label{main lemma}
Let $\B_G$ be the enlarged affine building of $G$, let $\C$ be a chamber of $\B_G$, then there exists a $\sigma$ stable apartment $\A$ that contains $\C$, and such an apartment is unique up to the action of an element of $H$ that fixes $\C$ pointwise.
\end{prop}

\begin{Remark}
    We can formulate the above proposition without the language of the building. The equivalent formulation is as follows. Let $I'$ be any Iwahori subgroup of $G$, there exists a $\sigma$ stable maximal split torus $T'$ such that the maximal compact subgroup $T'^0$ of $T'$ is contained in $I'$, i.e. $T'^0\subset I'$. Such a torus is unique up to conjugation by $H\cap I'$.

    This is an analogue of Lemma 2.4 of \cite{Helminck}. 
\end{Remark}

\begin{proof}[Proof of Proposition \ref{main lemma}]
The existence follows from Corollary \ref{2.23} in the present paper. 
Although this is not stated in \cite{delorm}, as explained in the preliminaries, it is a direct consequence of a main result of that paper.

We now prove uniqueness.

Let $\A_1,\A_2$ be two $\sigma$ stable apartments that contain $\C$. By item 3 of Theorem \ref{2.10}, there exists $g\in G$ such that $g\A_1=\A_2$ and $g$ fixes $S=\A_1\cap\A_2$ pointwise. 

Let $T_1$ be the maximal torus corresponding to $\A_1$ and let $T_2$ be the maximal torus corresponding to $\A_2$.

Consider $S$ as a subset of $\A_1$ and recall the notions from Notation \ref{2.7}, where the groups $U_S,P_S$ and $N_S$ are introduced.

By item 4 of Theorem \ref{2.11} and the Iwahori decomposition (Theorem \ref{2.16}) the group of elements that fix $S$ pointwise is $P_S=U_ST_1^0$. 

 $T_1^0$ acts trivially on $\A_1$ so we may assume that $g\in U_S$. The involution $\sigma$ maps $U_S$ to itself as it maps $S$ to itself and being an algebraic map sends unipotent elements to unipotent elements. So $\sigma(g)\in U_S$, hence $g^{-1}\sigma(g)\in U_S$. Furthermore, $g^{-1}\sigma(g)$ maps $\A_1$ to $\A_1$ so $g^{-1}\sigma(g)\in N_G(T_1)$.
 
 We conclude that $g^{-1}\sigma(g)\in N_G(T_1)\cap U_S$.

 Elements of $N_G(T_1)\cap U_S$ act trivially on $\A_1$. Therefore,  $N_G(T_1)\cap U_S\subset T_1^0\cap U_S$.

 The torus $T_1$ is split, so there is some $r\in \N$ and an isomorphism $\phi:T_1\rightarrow (F^\times)^r$. Using $\phi$ we can define an involution $\phi\circ\sigma\circ \phi^{-1}$ on $(F^\times)^r$. We abuse the notation and denote this involution also by $\sigma$.
 
 The map $\phi$ also induces an isomorphism between $T_1^0$ and $(\s^\times)^r$. 

  Let $\pi:\s\rightarrow k$ be the projection to the residue field and let $\s^1=\pi^{-1}(1)$.

 We have $\phi(U_S\cap T_1^0)\subset (\s^1)^r$, for a proof see for example Lemma 7.3.16 in \cite{BruhatTitsTA}, where a much more general result is proven.

 By Hensel's Lemma, any element of $(\s^1)^r$ has a square root inside $(\s^1)^r$. Thus, we can find $t\in (\s^1)^r$ such that $\phi(g^{-1}\sigma(g))=t^2$. We claim that $t\sigma(t)=1$. Denote $t'=t\sigma(t)$, we know that $(t')^2=1$. It is enough to check that $t'$ is 1 modulo any power of $\pi$. As $t'\in \s^1$, it is $1$ modulo $\pi$. Assume it is $1$ modulo $\pi^k$. Write $t'\equiv 1 +\pi^{k}b \text{ (mod } \pi^{k+1})$. We have $1\equiv t'^2\equiv 1+2b\pi^k \text{ (mod } \pi^{k+1})$, so we get that $t'$ is $1$ modulo $\pi^{k+1}$. By induction we are done.

 Let $g'=\phi^{-1}(t)\in T_1^0$. We have $g^{-1}\sigma(g)=(g')^2=g'\sigma(g')^{-1}$. Another way to write this is as $gg'=\sigma(gg')$. Thus $gg'\in H$. We know that $g'\in T^0$, so $g'$ acts trivially on $\A_1$. Thus, $gg'\A_1=g\A_1=\A_2$ and $gg'$ acts on $\C$ as $g$ does which means it fixes it pointwise and we are done.
 



\end{proof}

\subsection{I orbits on $X$}

To describe the $H$ orbits on $G/I$, we recall that $G/I$ can be identified with the set of $\Omega$ colored chambers in the affine building $\B_G$. We denote the set of $\Omega$ colored chambers by $C_\Omega(G)=\{(\C,o)|\C\subset\B_G,o\in\Omega\}$. The group $G$ acts on $C_\Omega(G)$ by $g(\C,o)=(g\C,\omega(g)o)$.

Let $\mathbf{A}_\sigma$ be the set of $\sigma$ stable apartments, $H$ acts on $\mathbf{A}_\sigma$ and by Theorem \ref{Fin_Sig_Apart} this action has finitely many orbits. Let $\A$ be a representative of an orbit and let $C_\A$ be the set of all $\Omega$ colored chambers contained in an apartment that is in the $H$ orbit of $\A$. Clearly $C_\A$ does not depend on the chosen representative and is stable under the action of $H$. By Proposition \ref{main lemma}, the sets of the form $C_\A$ are pairwise disjoint.

It is enough to describe the $H$ orbits on $C_\A$. 

\begin{prop}\label{3.2}
Let $\A$ be a $\sigma$ stable apartment in the affine building of $G$, let $T$ be the maximal split torus that corresponds to $\A$, define $W_{H,aff}(T)=Im(N_H(T)\rightarrow W_{aff}(T))$. 

 Fix a chamber $\C\subset \A$.

 The map from $W_{H,aff}(T)\backslash W_{aff}(T)$ to $H\backslash C_\A$ defined by $W_{H,aff}(T)w\mapsto H(w\C,\omega(w))$ is a bijection.

\end{prop}

\begin{proof}
It is clear that the map is well defined and surjective, we only have to check that the map is injective. Assume we have $w_1,w_2\in W_{aff}(T)$ such that we can find $h\in H$ with $hw_1\C=w_2\C$ and $\omega(hw_1)=\omega(w_2)$, we need to show that $w_1w_2^{-1}\in HT^0$.

 We have $hw_1w_2^{-1}w_2\C=w_2\C$. By Proposition \ref{main lemma} applied to the chamber $w_2\C$ we know that there is $h'\in H$ that fixes $w_2\C$ pointwise such that $hw_1w_2^{-1}\A=h'\A$. 

Therefore $h'^{-1}hw_1w_2^{-1}$ stabilizes $\A$ and fixes the chamber $w_2\C$ pointwise (pointwise, because $\omega(h')=\omega(hw_1w_2^{-1})=1)$. Thus $h'^{-1}hw_1w_2^{-1}\in T^0$. This completes the proof.
\end{proof}

For any $\sigma$ stable torus $T$ denote $W^H_{aff}(T)=W_{H,aff}(T)\backslash W_{aff}(T)$.

Let $\mathbf{A}=\{\A_0,...,\A_k\}$ be a set of representatives for the $H$ orbits on the $\sigma$ stable apartments (there is a finite number by Theorem \ref{Fin_Sig_Apart}).

Let $I$ be the Iwahori subgroup that fixes a chamber contained in $\A_0$, for any $i\neq0$ choose $g_i$ such that $g_i\A_0=\A_i$ and set $g_0=1$. Let $T$ be the torus that corresponds to $\A_0$.

From Proposition \ref{3.2} we deduce:

\begin{theorem}\label{3.3}
 The map from $\mathbin{\mathaccent\cdot\cup}_{0\leq i\leq k} W^H_{aff}(g_iTg_i^{-1})$ that sends $g\in W^H_{aff}(g_iTg_i^{-1})$ to $Hgg_iI$ is a bijection onto the $I$ orbits on $X$.
\end{theorem}

\subsection{Action of $W_{aff}$ on $H\backslash G/I$}\label{the action section}

Theorem \ref{3.3} allows us to define an action of the extended affine Weyl group of $G$ on the $I$ orbits on $X$. The group acts independently on each $W^H_{aff}(g_iTg_i^{-1})$ and on each one it acts via an isomorphism between the abstract group $W_{aff}$ and $W_{aff}(g_iTg_i^{-1})$, and the action of $W_{aff}(g_iTg_i^{-1})$ on its quotient. We do not have a canonical isomorphism between $W_{aff}$ and $W_{aff}(g_iTg_i^{-1})$. We use a chosen isomorphism between $W_{aff}$ and $W_{aff}(T)$ (given by Theorem \ref{2.3}) and conjugation by $g_i$ to identify $W_{aff}(T)$ and $W_{aff}(g_iTg_i^{-1})$.

We also give a more explicit description of this action. Denote by $\C_0$ the chamber fixed pointwise by $I$, we have $\C_0\subset\A_0$.

First, notice that Theorem \ref{3.3} implies the decomposition $G=\cup_{i=0}^k HW_{aff}(g_iTg_i^{-1})g_iI=\cup_{i=0}^k Hg_iW_{aff}(T)I$.

\begin{defn}\label{the action}
    Let $x\in G$ be a representative of a double coset $HxI$. Let $w\in W_{aff}\cong W_{aff}(T)$. There are $g_i$ and $w_1\in W_{aff}(T)$ such that $HxI=Hg_iw_1I$. We define the action as: $$HxI\times w=Hg_iw_1wI$$
\end{defn}

We need to check that this is well defined and does not depend on the choice of representatives.

\begin{prop}
    \begin{enumerate}
        \item The double coset $Hg_iw_1wI$ depends only on $HxI$ and $w$.
        \item Definition \ref{the action} gives a right action of $W_{aff}$ on $H\backslash G/I$.
        \item The defined action does not depend on the choice of the representatives $g_i$.
    \end{enumerate}
   
\end{prop}

\begin{proof}
    Let $g'\in G$ be such that $g'\A_0$ is a $\sigma$ stable apartment and let $w'_1\in W_{aff}$ be such that $Hg_iw_1I=Hg'w'_1I$. In order to prove items 1 and 3 it is enough to show that $Hg_iw_1wI=Hg'w'_1wI$. Notice that both $g_iw_1$ and $g'w'_1$ send $\A_0$ to a $\sigma$ stable apartment. Thus we can replace $g_iw_1$ with $g_i$ and $g'w'_1$ with $g'$. It is enough to prove that 
 if $Hg_iI=Hg'I$ then we also have $Hg_iwI=Hg'wI$. We now omit the $i$ index and write $g=g_i$.

 Let $\C$ denote the chamber fixed pointwise by $I$. Note that $\C\subset \A_0$.

 Let $h\in H$ such that $hg'I=gI$. Denote $\A=g\A_0$ and $\A'=hg'\A_0$, both are $\sigma$ stable apartments that contain $g\C$. By Proposition \ref{main lemma} applied to the chamber $g\C$ and the apartments $\A$ and $\A'$ we can find $h'\in H$ such that $h'g\C=g\C$ and $h'\A'=\A$. We also have $h'hg'I=h'gI=gI$.
 
 By replacing $h$ with $h'h$ we can assume that $g\A_0=hg'\A_0$.

 Consider $x=g^{-1}hg'$, we have $x\in I$ and $x\A_0=\A_0$, therefore $x$ must fix $\A_0$ pointwise. In particular $x$ fixes $w\C$ pointwise, which implies $w^{-1}xw=w^{-1}g^{-1}hg'w\in I$ or $hg'wI=gwI$. In particular $Hg'wI=HgwI$.

 Item 2 follows immediately from item 1.
\end{proof}



\begin{Remark}\label{3.5}
Notice that by the construction of the action it follows that for $x\in H\backslash G/I$ and $w\in W_{aff}$ we can find a $\Omega$ colored chamber $(\C_x,o_x)\in C_\Omega(G)$ whose $H$ orbit is $x$ and another $\Omega$ colored chamber $(\C_{x\times w},o_{x\times w})\in C_\Omega(G)$ whose $H$ orbits is $x\times w$ such that $\C_x$ and $\C_{x\times w}$ are contained in a single $\sigma$ stable apartment.
\end{Remark}

\section{The Key Computation}

Let $X=H\backslash G$ be a $p$-adic symmetric space. Recall that our goal is to describe the action of the Iwahori Hecke algebra $\mathcal{H}(G,I)$ on $S(X)^I$. By Theorem \ref{2.21}, it is enough to describe the action of $T_s$ for $s\in\Tilde{\Delta}$ and the action of $T_o$ for $o\in\Omega$.

We begin by describing the action of $T_o$ for $o\in\Omega$.

\begin{prop}\label{action of fund group}
    Let $o\in\Omega$ and let $x\in H\backslash G/I$. Let $T_o\in \mathcal{H}(G,I)$ be the characteristic function of $o$ and let $1_x\in S(X)^I$ be the characteristic function of $x$. Then $1_xT_o=1_{x\times o}$.
\end{prop}

\begin{proof}
    Let $m$ be a Haar measure on $G$ normalized so that $m(I)=1$.

    $$1_xT_o(y)=\int_{g\in G}1_x(yg^{-1})T_o(g)dm(g)$$

    If $1_xT_o(y)\neq 0$, then there is $g\in IoI$ such that $yg^{-1}\in HxI$, therefore $y\in HxIoI$. The group $\Omega$ normalizes $I$ so we obtain $y\in HxoI$. Thus, $1_xT_o$ is proportional to $1_{x\times o}$. If $y\in HxoI$ then $1_xT_o(y)=m(IoI)=1$, this completes the proof.
\end{proof}

Let $x\in  H\backslash G/I$ and let $s\in \Tilde{\Delta}$ be a simple reflection.

Let $T_s\in \mathcal{H}(G,I)$ be the generator that corresponds to $s$. Let $1_x\in S(X)^I$ be the characteristic function of $x$. 

In the remainder of this section, we translate the problem of computing $1_xT_s$ to a combinatorial problem about $H$ orbits of chambers in $\B_G$, the building of $G$.

We begin with the following lemma:
\begin{lemma}\label{3.13}
Let $\C$ be the chamber fixed by $I$ and let $\A$ be an apartment that contains $\C$. Let $s\in \Tilde{\Delta}$ be an element that induces a reflection over a facet of $\C$. 

The set of chambers $Is\C$ is the set of all chambers that contain $\C\cap s\C$ and are different from $\C$ and its size is $q$, the size of the residue field of $F$.  
\end{lemma}

\begin{proof}
Denote $f=\C \cap s\C$. The chamber $s\C$ contains $f$ and $I$ fixes $f$ so every chamber in $Is\C$ contains $f$.

Now, let $\C_1$ be a chamber that contains $f$ and is different from both $\C$ and $s\C$. By Theorem \ref{2.10} there is an apartment $\A'$ that contains both $\C$ and $\C_1$. Also there is $g\in G$ such that $g\A=\A'$ and such that $g$ fixes $\A\cap \A'$ pointwise. We have $\C\subset \A\cap\A'$ so $g$ fixes $\C$ pointwise and it sends $s\C$ to $\C'$. We found $g\in I$ such that $gs\C=\C_1$ which proves $\C_1\in Is\C$.

It is well known that the set of chambers that contain a given facet $f$ corresponds bijectively to $\p^1(k)$ (for example, this follows from Propositions \ref{root system} and \ref{affine to spherical building}). Excluding $\C$ there are $q$ chambers that contain $f$.

\end{proof}

Let $\C$ be the chamber fixed pointwise by $I$ and let $g\in G$ be a representative of the orbit $x$ such that $g\A_0$ is a $\sigma$ stable apartment. 

There exists $o\in \Omega$ such that $(g\C,o)$ is a $\Omega$ colored chamber whose $H$ orbit is $x$, and $(gs\C,o)$ is a $\Omega$ colored chamber whose $H$ orbit is $x\times s$, denote $f=g\C\cap gs\C$.

By Lemma \ref{3.13}, there are $q+1$ chambers that contain $f$. We will denote them by $\C_1,...,\C_{q+1}$.

The key formula is as follows.

\begin{prop}\label{key_computation}

Let $s\in \Tilde{\Delta}$, $x\in H\backslash G/I$, $g$ a representative of $x$ as above. Let $f=g\C\cap gs\C$ and let $\C_1,...,\C_{q+1}$ be the chambers that contain $f$. Denote $o=\omega(g)\in\Omega$.

Let $O_{f,o}$ be the set of $H$ orbits of $(\C_1,o),...,(\C_{q+1},o)$.

Define $D_{f,o}=\sum_{x\in O_{f,o}}1_x$ and 
let $\gamma_{f,g}=\#\{1\leq i\leq q+1|H(\C_i,o)=H(g\C,o)\}$.

Then we have:
\begin{center}
    $1_x(T_s+1)=\gamma_{f,g} D_{f,o}$
\end{center}
\end{prop}

\begin{proof}

First, observe that by the definition of convolution, $1_x(T_s+1)$ is supported on $HxIsI$. And that all the $H$ orbits of chambers in $HxIs\C$ can be represented by a chamber that contains $f$.

For each $1\leq i\leq q+1$ choose $g_1,...,g_{q+1}\in G$ such that $g_i\C_0=\C_i$ and $\omega(g_i)=\omega(g)$. Denote $x_i=Hg_i$.

We now compute $1_{x}T_s(x_i)$.
Let $m$ be a Haar measure on $G$ normalized such that $m(I)=1.$
A direct calculation gives:
$$1_{x}T_s(x_i)=\int_{g'\in G} 1_{x}(x_ig'^{-1})T_s(g')dm(g')=m(\{k\in IsI|g_ik^{-1}\in HgI\}).$$ 
On the other hand, 
$$m(\{k\in IsI|g_ik^{-1}\in HgI\})=m(IsI\cap Ig^{-1}Hg_i)=m(IsIg^{-1}_i\cap Ig^{-1}H).$$
To compute the intersection, we use Lemma \ref{3.13} to write $g_iIsI=\cup_{k\neq i}g_kI$

Taking inverses, we obtain $$IsIg^{-1}_i=\cup_{k\neq i}Ig^{-1}_k$$

Passing to measures, we get

$$m(IsIg^{-1}_i\cap Ig^{-1}H)=\sum_{k\neq i}m(Ig^{-1}_k\cap Ig^{-1}H)$$

Thus we obtain
 $$1_{x}T_s(x_i)=\sum_{k\neq i}m(Ig^{-1}_k\cap Ig^{-1}H)$$

We observe that, by the our choice of $g_{1},...,g_{q+1},$ for every $k$ we either have $Ig_k^{-1}\subset Ig^{-1}H$ or $Ig_k^{-1}\cap Ig^{-1}H=\emptyset.$

So the aforementioned sum is equal to

\[
    \beta_{g\C,g_i\C}=\sum_{k\neq i}\delta_{Hg_kI,HgI} \label{eq:special} \tag{*}
\]

Here, $\delta_{Hg_kI,HgI}$ is $1$ if $Hg_kI=HgI$ and zero otherwise.

The ''missing term'' in the formula 
$$1_{x}T_s(x_i)=\sum_{k\neq i}\delta_{Hg_kI,HgI}$$ motivates us to consider the quantity 
$1_{x}T_s(x_i)+1_{x}(x_{i}).$ which is independent of $i$ and equals $\gamma_{f,g}$.

Thus we obtain $1_x(T_s+1)(x_i)=\gamma_{f,g}$.

This computation also shows that $1_x(T_s+1)$ is supported on $\{x_1,...,x_{q+1}\}$ so we obtain $1_x(T_s+1)=\gamma_{f,g} D_{f,o}$
\end{proof}

\begin{Remark}
    Notice that the value $\gamma_{f,g}$ depends on the chamber $g\C$. For different chambers among $\C_1,...,\C_{q+1}$ the value may vary.
\end{Remark}

We see that to compute 
$1_xT_s$, for  $x\in  H\backslash G/I$ and $s\in \Tilde{\Delta}$ a simple reflection, we need to answer the following combinatorial question. Given $s$ as above, a chamber $\C$ and $o\in \Omega$ such that $(\C,o)$ corresponds to $x$, determine the sizes of the $H$ orbits of chambers that contain $f=\C\cap s\C$ with the color $o\in\Omega$. We address this in the following sections.

\section{Iwahori-Matsumoto type formulas}

 In this section, we prove a result
that generalizes the classical Iwahori-Matsumoto description given in Theorem \ref{2.20}. Thus, giving a partial solution to the problem of describing the action of $\mathcal{H}(G,I)$ on $S(X)^I$. This allows us to show that $S(X)^I$ is finitely generated over $\mathcal{H}(G,I)$. 

In the description of Iwahori-Matsumoto, a key role is played by the length function on the extended affine Weyl group. For our description, we will use a length function $l_{\sigma}$ on the set $H \backslash G/I.$

\subsection{The length function}
Our length function will be defined first on the chambers of the building. We begin by introducing this length function, $l_{\sigma}$.

For any two chambers $\C_1,\C_2\subset \B_{G}$, their distance $d(\C_1,\C_2)$ is defined to be the length of a minimal gallery between them.

\begin{defn}\label{l_sigma}
Let $\C$ be a chamber of $\B_{G}.$  Define $$l_\sigma(\C)=d(\C,\sigma(\C))$$ 
\end{defn}
\begin{prop}
$l_\sigma$ is constant on $H$ orbits.
\end{prop}

\begin{proof}
Let $h\in H$ and $\C$ be a chamber, then $l_\sigma(h\C)=d(h\C,\sigma(h\C))=d(h\C,h\sigma(\C))=d(\C,\sigma(\C))$.
\end{proof}

We use the notation $l_{\sigma}$ also for the induced length function on the $H$ orbits of chambers and on the $H$ orbits of $\Omega$ colored chambers. 

Under the identification between $H\backslash G/I$ with $H$ orbits of $\Omega$ colored chambers in the building, $l_\sigma$ induces a length function on $H\backslash G/I$ which we also denote by $l_\sigma$.

For $s\in \Tilde{\Delta}$ a simple reflection, let $T_s\in \mathcal{H}(G,I)$ be the corresponding element of the Hecke algebra. 

Now we can formulate an analogue of Theorem \ref{2.20}.

\begin{theorem}\label{simpleformula}
For $x\in H\backslash G/I$ let $1_x\in S(X)^I$ be the characteristic function of the orbit $x$. Let $s\in \Tilde{\Delta}$ be a simple reflection.

We have:
\begin{enumerate}
 \item If $l_\sigma(x\times s)>l_\sigma(x)$ then $1_xT_s=1_{x\times s}$.
    \item If $l_\sigma(x\times s)<l_\sigma(x)$ then $1_xT_s=(q-1)1_x+q1_{x\times s}$.
\end{enumerate}

\end{theorem}

\begin{Remark}
In the group case it is always true that $l_\sigma(x\times s) \ne l_\sigma(x)$, thus the theorem indeed generalizes the classical description of Iwahori and Matsumoto. However, for a general symmetric space, it may happen that $l_\sigma(x\times s)=l_\sigma(x)$. The general case will be addressed in Theorem \ref{3.25}.
\end{Remark}

For the proof of Theorem \ref{simpleformula} we will need the following lemma.

\begin{lemma}\label{3.9}
Let $x\in H\backslash G/I$ and let $s\in \Tilde{\Delta}$ be a simple reflection.
Let $\C$ be the chamber fixed by $I$. Choose $g\in G$ and choose a $\sigma$-stable apartment $\A$ such that $g$ is a representative of the orbit $x$ and $g\C,gs\C\subset\A$ (see Remark \ref{3.5}). Let $f=g\C\cap gs\C$.

Let $\Pi$ be the intersection of all $\sigma$ stable apartments that contain $f$.
Then we have the following criteria:
\begin{enumerate}
\item $l_\sigma(x \times s)>l_\sigma(x)$ if and only if $g\C\subset \Pi$.
    \item $l_\sigma(x\times xs)<l_\sigma(x)$ if and only if $gs\C\subset \Pi$.
    \end{enumerate}

\end{lemma}

\begin{proof}

Consider all the $\sigma$ stable apartments that contain $f$, there are three cases:
\begin{enumerate}
    \item All of them contain $g\C$.
    \item All of them contain $gs\C$.
    \item Neither case 1 nor case 2 happens.
\end{enumerate}

We will begin by showing that in case 3 we necessarily have $l_\sigma(x\times s) = l_\sigma(x)$ and thus there is nothing to prove.

The intersection of all $\sigma$ stable apartments that contain $f$ is a convex set in $\A$ (by Theorem \ref{2.13}) that contains $f$. In case 3, $\Pi$ does not contain any chamber in $\A$ that contains $f$ so $\Pi$ is contained in $P$, the affine hyper-plane spanned by $f$.

 $P$ is $\sigma$ stable as it is spanned by $\Pi$, the intersection of $\sigma$ stable apartments.

The fact that $P$ is $\sigma$ stable implies that $r_P$, the reflection across $P$, commutes with $\sigma$ on $\A$.

Let $(\C_i)_{i=1}^m$ be a minimal gallery that starts at $\C_1=g\C$ and ends at $\C_m=\sigma(g\C).$ Observe that by Theorem \ref{2.15} these chambers are all contained in the apartment $\A.$ Applying the reflection $r_{P}$ we obtain a minimal gallery $(r_P(\C_i))_{i=1}^m$  that starts with $r_P(g\C_1)=gs\C$ and ends with $r_P(\sigma(g\C))=\sigma(r_P(g\C))=\sigma(gs\C)$ which means that $l_\sigma(x\times s)=l_\sigma(gs\C)=l_\sigma(g\C)=l_\sigma(x)$.

This shows that in case 3 we have $l_\sigma(x\times s) = l_\sigma(x)$. 

To handle the other cases, it suffices to show that if $l_\sigma(gs\C)>l_\sigma(g\C)$ then case 1 occurs. That is, all $\sigma$ stable apartments that contain $f$ also contain $g\C$.
        
If $l_\sigma(gs\C)>l_\sigma(g\C)$ then $l_\sigma(gs\C)=l_\sigma(g\C)+2$ as $l_\sigma$ is constant modulo $2$ on each apartment. This means that we may choose a minimal gallery of chambers from $g\C$ to $\sigma(g\C)$ and then add $gs\C$ at the start and $\sigma(gs\C)$ at the end and get a minimal gallery of chambers from $gs\C$ to $\sigma(gs\C)$. By this argument and Theorem \ref{2.15} any $\sigma$ stable apartment that contains $gs\C$ has to contain $g\C$. Therefore either case 1 or case 3 happen and we know that it isn't case 3.

The case of $l_\sigma(gs\C)<l_\sigma(g\C)$ is treated in a similar way.

\end{proof}

Now we can deduce Theorem \ref{simpleformula}.

\begin{proof}[Proof of Theorem \ref{simpleformula}]
Let $\C$ be the chamber fixed by $I$, let $g\in G$ be a representative of the orbit $x$ and let $\A$ be a $\sigma$ stable apartment that contains $g\C$ and $gs\C$.

In case number 1, by Lemma \ref{3.9} all $\sigma$ stable apartments that contain $f=g\C\cap gs\C$ must contain $g\C$. 

Let $\C'$ be a chamber that contains $f$ and is different from $g\C$, let $\A'$ be a $\sigma$ stable apartment that contains $\C'$, it must contain $g\C$. By Proposition \ref{main lemma} there is $h\in H$ that fixes $g\C$ and satisfies $h\A'=\A$. This implies $h\C'=gs\C$.

This means that all chambers that contain $f$ and are different from $g\C$ are in the $H$ orbit of $gs\C$. The chamber $g\C$ itself is not in this $H$ orbit.

By Proposition \ref{key_computation} this yields $1_x(T_s+1)=1_{x\times s}+1_x$, or $1_xT_s=1_{x\times s}$.

The second case is similar.
\end{proof}

The next corollary provides a simple formula for the action of $T_w$ in some cases.

\begin{cor}\label{3.10}
Let $x\in H\backslash G/I$ be as in the previous theorem, let $w\in W_{aff}(T)$ and let $l:W_{aff}\rightarrow \mathbb{N}$ be the usual length function on $W_{aff}$ (see Section \ref{2.1} for a definition).

If $l_\sigma(x\times w)=2\cdot l(w)+l_\sigma(x)$ then 
\begin{equation*}
    1_xT_w=1_{x\times w}
\end{equation*}

\end{cor}

\begin{proof}
Write $w=os_1,...,s_{l(w)}$ as a product of simple reflections $s_i\in\Tilde{\Delta}$ and $o\in \Omega$. 

By Theorem \ref{2.20}  $T_w=T_oT_{s_1}...T_{s_{l(w)}}$. 

For any $1\leq j\leq l(w)$ we have $l_\sigma(x\times o \times s_1\times ....\times s_j)\leq l_\sigma(x)+2j$
. We must have equality; otherwise, the equality in the assumption would fail, as acting by each $s_i$ can increase the length by at most $2$.

By Theorem \ref{simpleformula}, Proposition \ref{action of fund group} and induction we obtain that $1_xT_oT_{s_1}...T_{s_j}=1_{x\times o \times s_1\times ....\times s_j}.$ For $j=l(w)$, we obtain the required result.
\end{proof}

This result is enough to prove that $S(X)^I$ is finitely generated over ${\mathcal H}(G,I).$

First we need the following two propositions.

\begin{prop}
    Let $T$ be a $\sigma$ stable maximal split torus of $G$ and let $\A$ be the apartment that corresponds to $T$. The torus $T$ acts on $\A$ by translations. Let $t\in T$ such that its action on $\A$ commutes with $\sigma$, then there is $t'\in T\cap H$ whose action on $\A$ is the same as $t$.
\end{prop}

\begin{proof}
    Let $\pi$ be a uniformizer of $\s$ and let $\Lambda(T)=\{\chi^\vee(\pi)|\chi^\vee\in X^*(T)\}$. The torus $T$ is a direct product of $\Lambda(T)$ and $T^0$ (see 3.6 in \cite{delorm}). Write $t=t_1t_2$ with $t_1\in\Lambda(T)$ and $t_2\in T^0$.

    Choosing a chamber $\C\subset\A$, we have $\sigma(t)\sigma(\C)=\sigma(t\C)=t\sigma(\C)$, so $t^{-1}\sigma(t)$ acts as a translation that fixes $\sigma(\C)$ and thus it acts trivially, hence $t^{-1}\sigma(t)\in T^0$. Thus $t_1^{-1}\sigma(t_1)\in T^0$. The set $\Lambda(T)$ is $\sigma$ stable so we obtain $t^{-1}_1\sigma(t_1)\in T^0\cap\Lambda(T)=1$ or $t_1\in H$. The elements $t$ and $t_1$ act the same way on $\A$ and the claim follows.
\end{proof}

\begin{prop}
    Let $\A$ be a $\sigma$ stable apartment and let $d\in\mathbb{N}$. The number of $H$ orbits of chambers $\C$ in $\A$ with $l_\sigma(\C)=d$ is finite. 
\end{prop}

\begin{proof}
    Let $T$ be the maximal torus which corresponds to $\A$. Notice that modulo the action of $T$ the number of chambers is finite (in fact equal to $\#W$). Thus modulo the action of $T$ the number of pairs of chambers $(\C_1,\C_2)$ in $\A$ such that $d(\C_1,\C_2)=d$ is finite. Let $\C_1,\C_2$ be two chambers in $\A$ such that $l_\sigma(\C_1)=l_\sigma(\C_2)$ and such that $(\C_1,\sigma(\C_1))$ and $(\C_2,\sigma(\C_2))$ are in the same $T$ orbit. There is $t\in T$ such that $t\C_1=\C_2$ and $t\sigma(\C_1)=\sigma(\C_2)$. This implies that the action of $t$ on $\A$ commutes with $\sigma$ and by the previous proposition $t$ may be replaced with an element of $T\cap H$. We obtain $\C_2=t\C_1$ with $t\in H$ so $\C_1$ and $\C_2$ are in the same $H$ orbit. Thus, the number of $H$ orbits of chambers $\C$ in $\A$ with $l_\sigma(\C)=d$ is bounded by the number of $T$ orbits of pairs of chambers $(\C_1,\C_2)$ with $d(\C_1,\C_2)=d$.
\end{proof}

\begin{theorem}\label{finite_generated}
Let $\mathbf{G}$ be a connected reductive group split over a $p$-adic field $F.$ Let $\mathbf{H}=\mathbf{G}^\sigma$ for some algebraic involution $\sigma$ and let $X=\mathbf{H}(F)\backslash \mathbf{G}(F)$ be a $p$-adic symmetric space. Then $S(X)^I$ is a finitely generated module over ${\mathcal H}(G,I).$
\end{theorem}

\begin{proof}
Consider the collection $\mathcal{Y}$ of 
$H$ orbits of chambers that contain a point fixed by $\sigma$. We claim that $\mathcal{Y}$ is finite and that the characteristic functions of these chambers (together with some $\Omega$ color) generate $S(X)^I$ over ${\mathcal H}(G,I).$

We first show that $\mathcal{Y}$ is finite. We know that there is a finite number of $H$ orbits of $\sigma$ stable apartments. It is enough to fix a $\sigma$ stable apartment $\A$ and to show that in $\A$ the number of $H$ orbits of chamber with a $\sigma$ fixed point is finite. If $\C$ is a chamber that has a $\sigma$ fixed point then $\sigma(\C)\cap\C\neq\emptyset$. It is clear that the distance between two intersecting chambers is bounded, and thus by the previous proposition there is a finite number of $H$ orbits of such chambers.

In order to prove the second part, we again fix a $\sigma$ stable apartment $\A$. Let $\C$ be a chamber in $\A$ without a $\sigma$ fixed point. It is enough to show that there is a minimal gallery between $\C$ and $\sigma(\C)$ such that the chamber $\C_1$ which appears immediately after $\C_0=\C$ satisfies $l_\sigma(\C_1)=l_\sigma(\C)-2$. If we show this, we can then conclude using Corollary \ref{3.10}, by induction on $l_\sigma(\C)$.

It is enough to show that there is a point $x\in \C$ such that the line connecting $x$ and $\sigma(x)$ intersects some chamber adjacent to $\C$. Assuming the contrary we can find $x$ such that $\frac{x+\sigma(x)}{2}\in \C$. This is a contradiction as $\C$ does not contain $\sigma$ fixed points.

\end{proof}

\begin{Remark}
Theorem \ref{finite_generated} can be deduced by combining the finite multiplicity result (Theorem 5.7) of \cite{Delormefinite} and the finite type results (Theorem B.0.2) of Appendix B of \cite{AGS}. We believe that our proof is more elementary.
\end{Remark}

\section{Involutions and parahoric group schemes}

To deal with the equality case, $l_\sigma(x\times s)=l_\sigma(x)$, we will use reduction to a symmetric space of a group of semisimple rank 1 over a finite field. To perform this reduction we consider, in this section, the action of $\sigma$ on certain parahoric group schemes. 

Assume that $\A$ is a $\sigma$ stable apartment of $\B_G$. Let $L\subset \A$ be a facet of some chamber in $\A$. Denote by $span(L)$ the minimal affine hyper plane that contains $L$ (this is an affine hyper plane, not a linear one). Assume that $span(L)$ is $\sigma$ stable, and let $L'=L\cup \sigma(L)$, clearly $L'$ is $\sigma$ stable.

Let ${\mathcal G}_L$, ${\mathcal G}_{\sigma(L)}$ and ${\mathcal G}_{L'}$ be the $\s$ group schemes which correspond to $L$, $\sigma(L)$ and $L'$ respectively (see Section \ref{s2.3}).
\begin{prop}
    $\sigma$ induces an involution $\sigma: {\mathcal G}_{L'}\rightarrow {\mathcal G}_{L'}$, $\sigma$ also induces a map ${\mathcal G}_{L}\rightarrow {\mathcal G}_{\sigma(L)}$.
\end{prop}

\begin{proof}
    We define the map on the open subscheme of ${\mathcal G}_{L'}$ isomorphic to ${\mathcal U}^+_{L'}\times \mathbf{T}\times {\mathcal U}^-_{L'}$ (see Theorem \ref{open subscheme}). The torus $\mathbf{T}$ is $\sigma$ stable and for any root $\alpha$, ${\mathcal U}_{\alpha_{L'}}$ is mapped to ${\mathcal U}_{\sigma(\alpha)_{L'}}$. The map $\sigma$ can be extended to ${\mathcal G}_{L'}$ in a unique way by Theorem 0.10 of \cite{Landvogt1995ACO}.

    The second part is done similarly.
\end{proof}

We have restriction maps $res^{L'}_L:{\mathcal G}_{L'}\rightarrow {\mathcal G}_{L}$ and $res^{L'}_{\sigma(L)}:{\mathcal G}_{L'}\rightarrow {\mathcal G}_{\sigma(L)}$ (see Proposition \ref{restriction maps}).

\begin{prop}\label{commute}
    We have a commutative diagram
\begin{center}
    
    \begin{tikzcd}
{\mathcal G}_{L'} \arrow{r}{\sigma} \arrow{d}{res^{L'}_{L}} & {\mathcal G}_{L'} \arrow{d}{res^{L'}_{\sigma(L)}} \\
{\mathcal G}_{L} \arrow{r}{\sigma} & {\mathcal G}_{\sigma(L)}
\end{tikzcd}

\end{center}
    
\end{prop}

\begin{proof}
    On the open subscheme isomorphic to ${\mathcal U}^+_{L'}\times \mathbf{T}\times {\mathcal U}^-_{L'}$ it is clear. Two maps that commute on an open subscheme must commute on the entire scheme.
\end{proof}

Let $C_L$ denote the set of all chambers containing $L$.

Let $P_L$ be the pointwise stabilizer of $L$. The group $P_L$ can be identified with the $\s$ points of ${\mathcal G}_{L}$.

Let $G_{L'}=\overline{\mathcal G}_{L'}/R_u(\overline{\mathcal G}_{L'})$ be the special fiber of  ${\mathcal G}_{L'}$ divided by its unipotent radical. The group $G_{L'}$ is the reductive quotient associated to $L'$. The involution $\sigma$ acts on $G_{L'}$.

We denote by $G^\sigma_{L'}$ the fixed points of $\sigma$ on $G_{L'}$. 

We will construct a bijection between the $H\cap P_L$ orbits on $C_L$ and the $G^\sigma_{L'}$ orbits on the flag variety of   $G_{L'}$.

\begin{prop}\label{easy direction}
    We have $H\cap P_L=H\cap P_{L'}$ and each element of this group projects to a $\sigma$ invariant element of $G_{L'}$. 
\end{prop}

\begin{proof}
    Clearly $H\cap P_{L'}\subset H\cap P_L$ and if $h\in H$ fixes $L$ pointwise it also fixes $\sigma(L)$ pointwise, hence $H\cap P_L=H\cap P_{L'}$. The second part is also clear.
\end{proof}

The converse also holds.

\begin{prop}\label{ hard direction}
    Any $\sigma$ invariant element of $G_{L'}$ can be lifted to an element of $H\cap P_{L}$.
\end{prop}

\begin{proof}
    Let $x$ be a $\sigma$ invariant element of $G_{L'}$. By Theorem \ref{Levi decomposition} There is a unique Levi factor $M\subset \overline{\mathcal G}_{L'}$ that contains $\mathbf{T}(k)$. The torus $\mathbf{T}(k)$ is $\sigma$ stable and the unipotent radical of $\overline{\mathcal G}_{L'}$ is also $\sigma$ stable, therefore $M$ must be $\sigma$ stable. There is $m\in M$ which projects to $x\in G_{L'}$. The element $\sigma(m)$ also projects to $\sigma(x)=x$ and $\sigma(m)\in M$, therefore we must have $m=\sigma(m)$. 

    $\overline{\mathcal G}_{L'}$ is smooth, and $\overline{\mathcal G}_{L'}^\sigma$ is also smooth (see for example Proposition 34 in \cite{Edixhoven1992}). By Hensel's lemma we can find a $\sigma$ invariant element of ${\mathcal G}_{L'}(\s)$ which projects to $m$. This element lies in $H\cap P_{L'}$. 
    
\end{proof}

Like in Proposition \ref{affine to spherical building}, we consider the following map, which we denote by $\overline{res}_L$. It is a map from $C_L$, the set of chambers that contain $L$, to the flag variety of $G_L$. Let $\C\in C_L$, define $\overline{res}_L(\C)=Im(\overline{res}^{\C}_{L})$, is a Borel of $G_L$. 

We denote by $B(K)$ the flag variety of a reductive group $K$ over the finite field $k$.

\begin{prop}\label{commutative diagram}
    The following diagram commutes:

\begin{center}

    \begin{tikzcd}
B(G_{L'}) \arrow{r}{\sigma} \arrow{d}{res^{L'}_{L}} & B(G_{L'}) \arrow{d}{res^{L'}_{\sigma(L)}} \\
B(G_{L}) \arrow{r}{\sigma} \arrow{d}{\overline{res}^{-1}_L} & B(G_{\sigma(L)}) \arrow{d}{\overline{res}^{-1}_{\sigma(L)}} \\ C_L \arrow{r}{\sigma} & C_{\sigma(L)}
\end{tikzcd}
\end{center}

\end{prop}

\begin{proof}
    The top square commutes by Proposition \ref{commute}. In order to check the commutativity of the bottom square it is enough to check that for any chamber $\C$ that contains $L$ the following diagram commutes:
\begin{center}
    
    \begin{tikzcd}
{\mathcal G}_{\C} \arrow{r}{\sigma} \arrow{d}{res^{\C}_{L}} & {\mathcal G}_{\sigma(\C)} \arrow{d}{res^{\sigma(\C)}_{\sigma(L)}} \\
{\mathcal G}_{L} \arrow{r}{\sigma} & {\mathcal G}_{\sigma(L)}
\end{tikzcd}

\end{center}

This can be checked on an open subscheme like in the proof of Proposition \ref{commute}.
    
\end{proof}

Consider the map from $B(G_{L'})$ to $C_L$ given as composition of the maps in the left column of the commutative diagram of the previous proposition. This map is an isomorphism by Propositions \ref{affine to spherical building} and \ref{face isomorphism}. Denote the inverse of this map by $b_L:C_L\rightarrow B(G_{L'})$.

Recall that we have $\omega:G\rightarrow \Omega$.

\begin{prop}\label{orbit correspondacne}
    The map $b_L$ sends $H\cap Ker(\omega)$ orbits to $G^\sigma_{L'}$ orbits. The inverse of $b_L$ sends $G^\sigma_{L'}$ orbits to $H\cap Ker(\omega)$ orbits. The map $b_L$ is a bijection from each orbit onto its image.
\end{prop}

\begin{proof}
    By Propositions \ref{easy direction} and \ref{ hard direction} we get that these maps interchange $H\cap P_L$ orbits and $G^\sigma_{L'}$ orbits. Thus we only need to prove that two chambers in $C_L$ are in the same $H\cap Ker(\omega)$ orbit if and only if they are in the same $H\cap P_L$ orbit. 

    By definition, $P_L\subset Ker(\omega)$. We only have to prove that if $\C,\C'$ are chambers that contain $L$ and $h\in H$ satisfies $h\C=\C'$ and $\omega(h)=1$ then $h$ fixes $L$ pointwise.

    By item 2 of Theorem \ref{2.10} there is an apartment $\A'$ that contains $\C$ and $\C'$. This apartment corresponds to some maximal torus $T'$.

    By item 5 of Theorem \ref{2.11} there is $n\in N_G(T')$ which acts on $\C'$ and on $\C$ like $h$. We have $\omega(h)=1$ so we obtain $\omega(n)=1$.

    The action of $n$ fixes $\C\cap \C'$ pointwise so $h$ fixes $L$ pointwise. Therefore $h\in P_L$ completing the proof.

    We know that $b_L$ is a bijection from $C_L$ to $B(G_{L'})$, in particular it is a bijection from each orbit onto its image.
\end{proof}

\begin{exmp}\label{section 6 examples}
    We give some examples to illustrate Proposition \ref{orbit correspondacne}. 
    \begin{enumerate}
        \item Let $G=SL_3(F)$ and let $\sigma:G\rightarrow G$ be given by $\sigma(g)=(g^t)^{-1}$. We have $H=G^\sigma=SO_3(F)$, a symmetric subgroup of $G$.  
        Let $T\subset G$ be the torus of diagonal matrices, $T$ is $\sigma$ stable.
        
        Let $\A$ be the apartment corresponding to $T$. $\sigma$ acts on $\A$ by reflection across the point corresponding to $SL_3(\s)$, we denote this point by $0$. If we take $L=0$ then we have $\sigma(L)=0$ and $L'=L$. In this case $G_{L'}=SL_3(k)$ and $G_{L'}^\sigma=SO_3(k)$. The map $b_L:C_L\rightarrow B(G_{L'})$ is the map that takes an Iwahori subgroup contained in $SL_3(\s)$ and reduces it modulo the uniformizer $\pi$ to get a Borel subgroup of $SL_3(k)$. The subgroup of $H$ consisting of elements that preserve $0$ is $SO_3(\s)$. The map $b_L$ sends $SO_3(\s)$ orbits to $SO_3(k)$ orbits.

        Now, let us choose a different $L$. Let $I$ be the Iwahori subgroup consisting of matrices of the form $\begin{pmatrix}
            \s^\times &\s &\s \\
            \pi\s & \s^\times & \s \\
            \pi \s & \pi\s & \s^\times
        \end{pmatrix}$. Let $P$ be the parahoric subgroup consisting of matrices in $SL_3(\s)$ of the form $\begin{pmatrix}
            \s&\s &\s \\
            \s & \s & \s \\
            \pi \s & \pi\s & \s^\times
        \end{pmatrix}$. Let $L$ be the facet fixed by $P$. The facet $\sigma(L)$ is the facet fixed by the parahoric subgroup consisting of matrices in $SL_3(\s)$ of the form $\begin{pmatrix}
            \s &\s &\pi\s \\
            \s & \s & \pi\s \\
            \s & \s & \s^\times
        \end{pmatrix}$. We have $L'=L\cup \sigma(L)$ and its pointwise stabilizer is equal to $P\cap \sigma(P)$. The group $G_{L'}$ is the group of matrices in $SL_3(k)$ of the form $\begin{pmatrix}
            k & k  &0 \\
            k & k & 0 \\
            0 & 0 & k^\times
        \end{pmatrix}$ which is equal to $SL_3(k)\cap (GL_2(k)\times GL_1(k))$.

        The map $b_L$ is a map between chambers that contain $L$ to the flag variety of $SL_3(k)\cap GL_2(k)\times GL_1(k)$ which is $\p^1(k)$. It sends an Iwahori subgroup contained in $P\cap \sigma(P)$ to its reduction modulo $\pi$, a Borel subgroup of  $SL_3(k)\cap GL_2(k)\times GL_1(k)$.

        The involution $\sigma$ acts on $SL_3(k)\cap GL_2(k)\times GL_1(k)$ as transpose inverse and its fixed points are $SL_3(k)\cap O_2(k)\times O_1(k)$. The $H$ orbits on $C_L$ are mapped to $SL_3(k)\cap O_2(k)\times O_1(k)$ orbits on $\p^1(k)$.

        \item We give one more example, this time in type C. 
        We take $G=Sp_4(F)$ and we use the following block representation of  $Sp_4$. A block matrix $\begin{pmatrix}
            A & B\\
            C& D
        \end{pmatrix}\in Sp_4(F)$ if and only if $A^tC=C^tA, B^tD=D^tB$ and $A^tD-B^tC=I_2$. Here, $I_2$ is the identity $2\times 2$ matrix.
                
        Let $H=GL_2(F)$ the fixed points of an inner involution $\sigma$. The involution $\sigma$ is defined as conjugation by $\begin{pmatrix}
            1 & 0 &0 &0\\
            0 & 1 & 0& 0\\
            0& 0& -1 &0\\
            0&0&0& -1
        \end{pmatrix}$.

        An element $g\in H$ is embedded as $\begin{pmatrix}
            g & 0\\
            0 & (g^t)^{-1}
        \end{pmatrix}\in Sp_4(F)$.

        Let $T$ be the torus of diagonal matrices in $H$. We can think about it as a maximal torus of $G$ using the embedding of $H$ in $G$. Let $\A$ be the apartment corresponding to $T$. On this apartment, $\sigma$ acts trivially.

        We have the Iwahori subgroup consisting of integer matrices of the form $\begin{pmatrix}
            A & B\\
            C& D
        \end{pmatrix}\in Sp_4(\s)$, with $A,D$ upper triangular mod $\pi$ and $C$ zero mod $\pi$. We denote this Iwahori subgroup by $I$. One of the parahoric subgroups that contains $I$ is the subgroup $P$ of matrices in $Sp_4(\s)$ of the form $\begin{pmatrix}
            \s^\times & \s &\s&\s\\
            \pi\s& \s & \s &\s \\
            \pi \s & \s & \s &\s \\
            \pi \s & \s & \s &\s
        \end{pmatrix}$. Let $L$ be the facet fixed by $P$, we have $L=\sigma(L)=L'$ as $\sigma$ acts trivially on $\A$. We have $G_{L'}=GL_1\times SL_2$ and $\sigma$ acts on $G_{L'}$ as an inner involution of $SL_2$. Thus, $G_{L'}^\sigma=GL_1\times GL_1$. The map $b_L$ gives a bijection between Iwahori subgroups contained in $P$ and Borel subgroups of $GL_1\times SL_2$. $H$ orbits are mapped to $GL_1\times GL_1$ orbits.
    \end{enumerate}
\end{exmp}

\begin{defn}
    We can define a length function $l_\sigma$ on the spherical building of $G_{L'}$ in the same way we did for the affine building $\B_G$.
\end{defn}

\begin{prop}
    For any chamber $\C\in C_L$ the number $l_\sigma(b_L(\C))-l_\sigma(\C)$ depends only on $L$ and not on $\C$.
\end{prop}

\begin{proof}
    It is enough to check this for chambers which intersect at a facet of codimension one. This follows from results in the next section, in particular from Proposition \ref{7.3} and Corollary \ref{3.24}.   
\end{proof}

\begin{Remark}
    The $l_\sigma$ lengths on the spherical building can be used to describe the Bruhat order of Borel orbit closures in the Zariski topology. Using the previous proposition we can define a Bruhat order on $C_L$.
\end{Remark}

\section{Determining the structure constants $\gamma_{f,g}$}\label{s7}

In this section, we focus on calculating the constants $\gamma_{f,g}$ occurring in the formula $$1_{x}(T_{s}+1)=\gamma_{f,g}D_{f,o}$$ as proved in Proposition \ref{key_computation}.

Let $s\in \tilde{\Delta}$ be a simple reflection. Let $\C$ be the chamber fixed by $I$ and let $g\in G$ be a representative of the orbit $x\in H\backslash G/I.$ Let $\A$ be a $\sigma$ stable apartment that contains $g\C$ and $gs\C$. 

Consider the facet $f=g\C\cap gs\C$.

In the case $l_\sigma(x\times s)\neq l_\sigma(x)$ we know the value of $\gamma_{f,g}$. It equals $1$ if $l_\sigma(x\times s)> l_\sigma(x)$ and equals $q$ if $l_\sigma(x\times s) < l_\sigma(x)$.

We now assume that $l_\sigma(x\times s)= l_\sigma(x)$. This implies that the affine hyper plane spanned by $f$ is $\sigma$ stable (this was explained during the proof of Lemma \ref{3.9}). Thus, $f$ satisfies the assumption of the previous section. 

Let $f'=f\cup \sigma(f)$. Denote by $C_f$ the set of all chambers in $\B_G$ that contain $f$. Let $G_{f'}$ be the reductive quotient associated to $f'$, it is a reductive group over the residue field $k$. The map $\sigma$ acts on $G_{f'}$ as was shown in the previous section. Let $B(G_{f'})$ be the flag variety of $G_{f'}$.

Let $b_f:C_f\rightarrow B(G_{f'})$ be the map constructed in the previous section.

\begin{prop}\label{6.1}
The number $\gamma_{f,g}$ is equal to the size of the $G^\sigma_{f'}$ orbit of $b_f(g\C)$ in the flag variety of $G_{f'}$. 
\end{prop}

\begin{proof}
By definition, $\gamma_{f,g}$ is equal to the number of chambers in $C_f$ in the $H\cap Ker(\omega)$ orbit of $g\C$. By Proposition \ref{orbit correspondacne} this is equal to the size of the $G^\sigma_f$ orbit of $b_f(g\C)$ in the flag variety of $G_f$.
\end{proof}

\begin{prop}
    The semisimple rank of $G_{f'}$ is 1.
\end{prop}

\begin{proof}
    By Proposition \ref{root system} the root system of $G_{f'}$ consists of the roots constant on $f$, $f$ is a facet of codimension 1 so there is a single root (up to a sign) in the root system of $G_{f'}$. Therefore the semisimple rank of $G_{f'}$ is 1. 
\end{proof}

\begin{exmp}
    In Example \ref{section 6 examples} we saw several examples of this when $L$ had codimension 1. In the first example we had $G_{f'}\cong SL_3\cap GL_2\times GL_1$ and in the second example we had $G_{f'}\cong SL_2\times GL_1$. In both examples $G_{f'}$ has semisimple rank 1.
\end{exmp}

Thus, to compute the value of $\gamma_{f,g}$ we need to compute the size of an orbit of a symmetric subgroup of a split reductive group of semisimple rank 1, on the flag variety of this reductive group.

\begin{prop}\label{rank1 groups}
    All split reductive groups of semisimple rank 1 are isomorphic to one of the following groups:
    \begin{enumerate}
        \item $SL_2\times\mathbb{G}^r_m$
        \item $PGL_2\times\mathbb{G}^r_m$
        \item $GL_2\times\mathbb{G}^r_m$
    \end{enumerate}

    See 20.3.3 in \cite{Milne_2017}.
\end{prop}

The flag variety of any group of semi simple rank 1 is $\mathbb{P}^1$. In order to compute the possible values of $\gamma_{f,g}$ we have to compute the size of the orbit of $G^\sigma_f$ on $\mathbb{P}^1$.

Let $K$ be the $k$ points of one of the groups from the previous proposition, let $\sigma$ be an algebraic involution on $K$ and let $M=K^\sigma$. Let $Z=Z(K)^0$ be the connected component of the center of $K$.

$\sigma$ induces an involution on $K/Z$, and $M$ projects into the fixed points of this involution on $K/Z$. 

Note that the action of $K$ on its flag variety factors through $K/Z$. In the first case of Proposition \ref{rank1 groups} we have $K/Z=SL_2(k)$, in the two other cases we have $K/Z=PGL_2(k)$. Thus, we are interested in symmetric subgroups of $PGL_2(k),SL_2(k)$.

\begin{prop}
    The symmetric subgroups of $SL_2(k)$ up to conjugation are:
    \begin{enumerate}
        \item $SL_2(k)$.
        \item $T$ a split torus.
        \item $\Tilde{T}$ a non-split torus.
    \end{enumerate}

    The symmetric subgroups of $PGL_2(k)$ up to conjugation are:
    \begin{enumerate}
        \item $PGL_2(k)$.
        \item $N_{PGL_2(k)}(T)$ a normalizer of a split torus.
        \item $N_{PGL_2(k)}(\tilde{T})$ a normalizer of a non-split torus. 
    \end{enumerate}
    See \cite{AGB-120006486}.

\end{prop}

In the first two cases of Proposition \ref{rank1 groups}, all the $\sigma$ stable elements of $K/Z$ lift to $M$. 

In the third case this might not be true. However, if we restrict to $SL_2\subset GL_2\times \mathbb{G}^r_m$ it is true. That is, all $\sigma$ stable elements of $SL_2/Z$ do lift to $SL_2\cap M$. This means that $M/Z$ must contain the index 4 subgroup $T^2\subset N_{PGL_2(k)}(T)$ both in the case of a split and in the case of non-split torus.

Consider $\tilde{M}=Im(M\rightarrow Aut(\mathbb{P}^1))$ and its orbits on the flag variety of $K$, the possibilities are summarized by the next proposition:

\begin{prop}\label{7.3}
    Let $\tilde{M}$ be the image of $M$ in the automorphism group of the flag variety of $K$, which is isomorphic to $PGL_2(k)$.

    Let $T_{sp}$ be a split torus of $PGL_2(k)$ and let $w\in N_{PGL_2(k)}(T_{sp})$ be an element that normalizes $T_{sp}$ but is not in $T_{sp}$. 

    Let $T_n$ be a non-split torus of $PGL_2(k)$ and let $\Tilde{w}\in N_{PGL_2(k)}(T_n)$ be an element that normalizes $T_n$ but is not in $T_n$.

    The options for $\tilde{M}$ up to conjugation, as well as the options for the numbers and sizes of its orbits on $\mathbb{P}^1$ are given by Table \ref{tab:my_label}.
    \begin{table}[]
        \caption{Subgroups $\Tilde{M}$ of $PGL_2(k)$ which can be obtained as the image of $M$. Together with the number and sizes of their orbits on $\p^1$. }
            \begin{center}

    \begin{tabular}{|c|c|c|}
    \hline
    $\Tilde{M}$ & number of orbits & sizes of orbits \\
    \hline
        $PSL_2(k)$ &1 &$q+1$  \\
        \hline
        $PGL_2(k)$ &1 &$q+1$ \\
         \hline
         $T_{sp}$& 3 &$1,1,q-1$ \\
         \hline
          $T_{sp}^2$& 4 &$1,1,\frac{q-1}{2},\frac{q-1}{2}$\\
         \hline
         $T_{sp}\cup wT_{sp}$ & 2 & $2,q-1$\\
         \hline
         $wT_{sp}^2\cup T_{sp}^2$& 3 & $2,\frac{q-1}{2},\frac{q-1}{2}$\\
         \hline
         $wT_{sp}\cup T_{sp}^2$& 2 & $2,q-1$\\
         \hline
         
         \hline
         $T_n$& 1 &$q+1$ \\
         \hline
          $T_n^2$& 2 &$\frac{q+1}{2},\frac{q+1}{2}$\\
         \hline
         $T_n\cup \Tilde{w}T_n$&  1 & $q+1$\\
         \hline
         $\Tilde{w}T_n^2\cup T_n^2$& 2 & $\frac{q+1}{2},\frac{q+1}{2}$\\
         \hline
         $\Tilde{w}T_n\cup T_n^2$& 1 & $q+1$\\
         \hline

    \end{tabular}
    \label{tab:my_label}
            \end{center}

        \end{table}

        \end{prop}

        \begin{proof}
            This follows from a simple computation in each case given by the previous proposition. As well as in the case $K=GL_2(k)\times (k^\times)^r$ and $T^2\subset \Tilde{M}\subset N_{PGL_2(k)}(T)$ for $T$ a torus which may be either split or a non-split.
        \end{proof}

Recall that the structure constant $\gamma_{f,g}$ is equal to the size of an orbit of $\Tilde{M}$ on $\mathbb{P}^1$.

A corollary we mention here is the fact that the structure constants $\gamma_{f,g}$ are polynomial in $q$. In fact we have:
\begin{cor}\label{3.21}
The possible values of $\gamma_{f,g}$ are $1,2,\frac{q-1}{2},q-1,q,q+1$ and $\frac{q+1}{2}$.
\end{cor}

\begin{proof}
By Theorem \ref{simpleformula} if $l_\sigma(gs\C)\neq l_\sigma(g\C)$ then $\gamma_{f,g}$ is either $1$ or $q$.
By Proposition \ref{7.3} if $l_\sigma(gs\C)=l_\sigma(g\C)$ then $\gamma_{f,g}$ is one of $1,2,\frac{q-1}{2},q-1,q+1$ and $\frac{q+1}{2}$.
\end{proof}

\begin{Remark}
All the values mentioned in Corollary \ref{3.21} occur.

The values $1,q,q+1$ are obtained in the case $G=SL_4$ and $H=Sp_4$. 

The value $\frac{q-1}{2}$ is obtained in the case $G=SL_2$, $H=T$ a maximal split torus.

The value $q-1$ is obtained in the case $G=GL_2$ and $H=T$ a maximal split torus. 

The value $2$ is obtained in the case $G=PGL_2$ and $H=N_G(T)$ a normalizer of a maximal split torus.

The value $\frac{q+1}{2}$ is obtained in the case $G=SL_2$ and $H=\Tilde{T}$ a non-split torus.
\end{Remark}

 We will now present a way to determine $\gamma_{f,g}$ directly using the $l_\sigma$ lengths of the chambers that contain $f$.

\begin{prop}\label{3.23}
Let $\C_1,...\C_{q+1}$ be the collection of chambers that contain $f$. The set 

$\{l_\sigma(\C_1),...,l_\sigma(\C_{q+1})\}$ is of size at most 2. Furthermore, the difference between different $l_\sigma$ lengths is at most 2 if $l_\sigma(x)\neq l_\sigma(x\times s)$ and at most 1 if $l_\sigma(x)= l_\sigma(x\times s)$.
\end{prop}

\begin{proof}
If $l_\sigma(x)< l_\sigma(x\times s)$ then by Lemma \ref{3.9} every chamber in $\{\C_1,...\C_{q+1}\}$ which is not equal to $g\C$ is contained in a $\sigma$ stable apartment, $\A$, that satisfies $g\C\subset\A$. Thus, by Proposition \ref{main lemma} all chambers in $\{\C_1,...\C_{q+1}\}$ different from $g\C$ are in the same $H$ orbit and hence have the same $l_\sigma$ length. This $l_\sigma$ length is equal to $l_\sigma(x\times s)=l_\sigma(x)+2$. The same argument also works in the case $l_\sigma(x)> l_\sigma(x\times s)$.

Now assume that $l_\sigma(x)= l_\sigma(x\times s)$. Let $\C'$ be a chamber that contains $f$ and let $\C''$ be a chamber that contains $\sigma(f)$. There is an apartment $\A$ that contains both $\C'$ and $\C''$. Denote by $\overline{f}$ the affine subspace spanned by $f$ inside $\A$. There are two cases, either $\C'$ and $\C''$ are on the same side of $\overline{f}$ or they are on different sides. This distinction determines $d(\C',\C'')$, in the case they are on different sides their distance is larger by $1$ than the distance in the case they are on the same side.

For any chamber $\C'$ that contains $f$, $\sigma(\C')$ contains $\sigma(f)$. Therefore there are two possible values for $l_\sigma(\C')=d(\C',\sigma(\C'))$ and the difference between them is $1.$

\end{proof}

\begin{Remark}\label{r7.9}
    Notice that we proved that if $l_\sigma(x)= l_\sigma(x\times s)$ then $l_\sigma(g\C)$ may take one of two values, those values depend only on $f$. $l_\sigma(g\C)$ attains its minimal value if and only if $\sigma(g\C)$ and $g\C$ are on the same side of $f$ inside some apartment. 
\end{Remark}

\begin{prop}
    Assume $f$ is a facet of codimension 1, let $f'=f\cup \sigma(f)$ and let $G_{f'}$ be the reductive group of semisimple rank 1 defined over $k$ associated to $f'$. Recall we have an action of $\sigma$ on $G_{f'}$. Let $b_f$ be the map from $C_f$, the set of chambers that contain $f$, to $B(G_{f'})$, the flag variety of $G_{f'}$. Then $g\C$ and $\sigma(g\C)$ are on the same side of $f$ in some apartment if and only if $b_L(g\C)$ is $\sigma$ stable.
\end{prop}

\begin{proof}
    Recall the commutative diagram from Proposition \ref{commutative diagram}.

\begin{center}   
    
    \begin{tikzcd}
B(G_{f'}) \arrow{r}{\sigma} \arrow{d}{res^{f'}_{f}} & B(G_{f'}) \arrow{d}{res^{f'}_{\sigma(f)}} \\
B(G_{f}) \arrow{r}{\sigma} \arrow{d}{\overline{res}^{-1}_f} & B(G_{\sigma(f)}) \arrow{d}{\overline{res}^{-1}_{\sigma(f)}} \\ C_f \arrow{r}{\sigma} & C_{\sigma(f)}
\end{tikzcd}
\end{center}

    Let $p\in B(G_{f'})$, $p$ is fixed by $\sigma$ if and only if $\sigma(res^{f'}_f(p))=res^{f'}_{\sigma(f)}(p)$. Let $\C_p\in C_f$ be the chamber that corresponds to $res^{f'}_f(p)$ and let $\C_p'\in C_{\sigma(f)}$ the chamber that corresponds to $res^{f'}_{\sigma(f)}(p)$. We know that $\sigma(p)=p$ if and only if $\sigma(\C_p)=\C_p'$, thus it is enough to prove that $\C_p$ and $\C_p'$ are always on the same side of $f$. 
    
    We lift $p$ to a subgroup of $P_{f'}$, we denote it by $S$. We have $S\subset P_{f'}={\mathcal G}_{f'}(\s)$. The restriction maps are just the restriction to the groups $P_f$ and $P_{\sigma(f)}$ that contain $P_{f'}$. The passage to chambers in $C_f$ (or $C_{\sigma(f)}$) is given by taking the points in the building fixed by $S$ and finding inside them a chamber that contains $f$ (respectively $\sigma(f)$). Therefore, we only need to show that if $\C'$ is a chamber that contains $f$, $\A$ is an apartment that contains both $\C'$ and $\sigma(f)$, and $g'\in S$ fixes both $\C'$ and $\sigma(f)$ pointwise then $g'$ fixes a chamber in $\A$ that contains $\sigma(f)$ on the same side of $f$ as $\C$. This follows from the fact that the set of elements in $\A$ fixed by $g'$ is convex (see Theorem \ref{2.13} and part 5 of  Theorem \ref{2.11}). 
\end{proof}

\begin{cor}\label{3.24}
Assume $l_\sigma(x)=l_\sigma(x\times s)$ and that not all of $l_\sigma(\C_1),...,l_\sigma(\C_{q+1})$ are equal. Let $b_f$ be the map from $C_f$, the set of chambers that contain $f$, to $B(G_{f'})$ the flag variety of $G_{f'}$. Let $\C_i$ be a chamber that contains $f$. The value $l_\sigma(\C_i)$ is minimal among $l_\sigma(\C_1),...,l_\sigma(\C_{q+1})$ if and only if $b_f(\C_i)$ is $\sigma$ stable.
\end{cor}

\begin{proof}
    This follows immediately from the previous proposition and Remark \ref{r7.9}.
\end{proof}

Using the previous result we can determine $\gamma_{f,g}$ from the $l_\sigma$ lengths of the orbits of the chambers that contain $f$.

As explained above, the calculation of $T_{s}1_{x}$ was reduced to the calculation of $\gamma_{f,g}$. To determine this number we need to know further information regarding the lengths of certain chambers.
Let $\C_1,...,\C_{q+1}$ be all the chambers that contain $f$.
Let $\s_f$ be the set of their $H$ orbits. By Proposition \ref{3.23} there are at most two possibilities for the $l_\sigma$ lengths of orbits in $\s_f$. Denote by $n_{max},n_{min}$ be the number of orbits in $\s_f$ of maximal/minimal $l_\sigma$ length respectively. Let $\delta_{x,max}$ be 1 if $l_\sigma(g\C)$ is maximal among $l_\sigma(\C_1),...,l_\sigma(\C_{q+1})$ and $0$ otherwise. The value $\delta_{x,min}$ is defined in a similar way. 

\begin{theorem}\label{3.25}
We have $n_{max},n_{min}\in\{1,2\}$.
If $l_\sigma(x)= l_\sigma(x\times s)$ then $$\gamma_{f,g}=\frac{q-1}{n_{max}}\delta_{x,max}+\frac{2}{n_{min}}\delta_{x,min}$$

\end{theorem}

\begin{proof}

By Propositions \ref{3.24}, \ref{easy direction} and \ref{ hard direction} we know that the number of orbits of minimal/maximal $\sigma$ length is equal to the number of $G^\sigma_f$ orbits on the flag variety of $G_{f'}$ of minimal/maximal size. Now a simple check of the cases in Proposition \ref{7.3} verifies the result.

\end{proof}

\section{A Module over the generic Hecke algebra}
    
In this section, we will construct an abstract module $M_t$ over the abstract Iwahori Hecke algebra $\mathcal{H}_t$, which specializes to $S(X)^I$ as a module over $\mathcal{H}(G,I)$.

\begin{defn}
The Iwahori hyper-graph $\Gamma_X$ of $X$ is defined as follows:
Its vertices correspond to the $H$ orbits on $G/I$. The vertex that corresponds to $x\in H\backslash G/I$ is denoted by $v_x$. For every facet of codimension 1 in $\B_G$, the building of $G$, and for every $o\in \Omega$, we have a hyper-edge that contains all the $H$ orbits of chambers that contain the facet $f$ with the color $o$.

There is an action of $W_{aff}$ on the vertices (but not on the graph), and there is a function $l_\sigma$ from the set of vertices to $\N$. This function corresponds to the length function $l_\sigma$ on the orbits.
\end{defn}

\begin{prop}
The possible sizes for the hyper-edges are $1,2,3$ and $4$.
\end{prop}

\begin{proof}
This follows from Proposition \ref{3.25}.
\end{proof}

\begin{defn}
    
Let $t$ be a variable and let $\mathcal{H}_{t}=\mathcal{H}_{t}(W_{aff},\tilde{\Delta})$ be the generic Hecke algebra of $(W_{aff},\tilde{\Delta})$ (defined in many sources, for example \cite{Solleveld2021}). It is defined by generators and relations as follows:

It is generated, as an algebra, by $t$, by $T_s\in \mathcal{H}_{t}(W_{aff},\tilde{\Delta})$ for any $s\in \tilde{\Delta}$ and by $T_o\in \mathcal{H}_{t}(W_{aff},\tilde{\Delta})$ for $o\in \Omega$. The element $t$ is invertible and commutes with all $T_s$ and $T_o$ and the relations the $T_s, T_o$ satisfy are:

 \begin{enumerate}
            \item $(T_s+1)(T_s-t)=0$
            \item if $s_1,s_2\in\Tilde{\Delta}$, $s_1\neq s_2$ and $(s_1s_2)^m=1$ then $T_{s_1}T_{s_2}T_{s_1}...=T_{s_2}T_{s_1}T_{s_2}...$  where each product contains $m$ elements.
            \item For $o_1,o_2\in \Omega$ we have $T_{o_1}T_{o_2}=T_{o_1o_2}$.
            \item For $o\in \Omega$ and $s\in\Tilde{\Delta}$ we have $T_o T_s=T_{o s o^{-1}}T_o$.
         \end{enumerate}

\end{defn}

When we specialize $t$ to be the size of the residue field $k$, $t=q=\#k$. The algebra $\mathcal{H}_{t}(W_{aff},\tilde{\Delta})/(t-q)$ is isomorphic to the Iwahori-Hecke algebra $\mathcal{H}(G,I)$. 

Note that $\mathcal{H}_{t}(W_{aff},\tilde{\Delta})/(t-1) \simeq \mathbb{C}[W_{aff}].$

The calculation of the structure constants $\gamma_{f,g}(q)$ and the fact that they depend in a polynomial way on $q$ shows that the modules $S(X)^{I}$, over the algebras $\mathcal{H}(G,I)$, admit a similar generic version. 

Indeed, we show below how to use the hyper-graph $\Gamma_X$ in order to construct an abstract module $M_t=M_{t}(\Gamma_X)$ over $\mathcal{H}_{t}(W_{aff},\tilde{\Delta})$. 

\begin{defn}
Let $M_{t}(\Gamma_X)$ be the free $\mathbb{C}[t]$ module generated by the vertices of $\Gamma_X$. For $s\in \Tilde{\Delta}$ a simple reflection, we define an action of $T_s$ on $M_{t}(\Gamma_X)$ as follows.

Let $v$ be a vertex of $\Gamma_X$. Let $e_{v,v\times s}$ be the unique hyper-edge that contains $v$ and $v\times s$. 

Then $$vT_s=-v+\gamma_{v,v\times s}(t)\sum_{v'\in e_{v,v\times s}}v'$$  where $\gamma_{v,v\times s}(t)$ is a polynomial in $t$ determined according to Theorem \ref{3.25} with $v$ playing the role of $x$, $e_{v,v\times s}$ the role of $\s_f$ and $t$ the role of $q$.

We also define an action of $T_o$ for $o\in \Omega$. $$ vT_o= v\times o$$
\end{defn}

We call $M_{t}(\Gamma_X)$ the generic Hecke module attached to the symmetric space $X$. 
The connection between the generic Hecke module $M_{t}(\Gamma_X)$  and the Hecke module $S(X)^I$ is given by the theorem below.

\begin{theorem}\label{generic}

Consider the map $\alpha: M_{t}(\Gamma_X) \to S(X)^I$ given by $$\alpha(v_x)=1_x, \alpha(t)=q$$  
For every $x\in H\backslash G/I$, the map $\alpha$ sends the generator of $M_t(\Gamma_X)$ corresponding to the vertex $v_x$ to the characteristic function of $x$. Then 
\begin{itemize}
    \item $\alpha$ induces an isomorphism $\alpha_{q}$ of vector spaces between $M_{t}(\Gamma_X)/(t-q)$ and $S(X)^I$.
    
    \item $\alpha$ commutes with the action of $T_s$ for any $s\in \Tilde{\Delta}$ and with the action of $T_o$ for any $o\in \Omega$.
    \item The actions of $T_s$ and $T_o$ for $s\in \Tilde{\Delta}$ and $o\in \Omega$ define an action of the generic Iwahori Hecke algebra $\mathcal{H}_{t}(W_{aff},\tilde{\Delta})$ on $M_{t}(\Gamma_X)$.
    \item $\alpha_{q}:M_{t}(\Gamma_X)/(t-q) \to S(X)^{I}$ is an isomorphism of $H(G,I)$ modules.
\end{itemize}
\end{theorem}

\begin{proof}
The first two assertions are trivial and it is clear that the fourth assertion follows from the previous three.

Thus we only need to prove the third assertion. We have to show that for any vertex $v\in \Gamma_X$ the relations defining $\mathcal{H}_{t}(W_{aff},\tilde{\Delta})$ hold under the actions we have defined on the vertices. 

These relations are polynomial in $t$, so it is enough to show that they hold for infinitely many values of $t$. By construction they hold for $t=q$. Let $r$ be an odd integer and let $F_r$ be an unramified extension of $F$ of order $r$, denote by $\s_r$ its ring of integers and by $k_r$ its residue field. Denote by $X(F_r)=\mathbf{H}(F_r)\backslash\mathbf{G}(F_r)$, it is enough to show that $\Gamma_{X(F_r)}=\Gamma_X$ for any odd $r$.

Let $\B_G'$ be the building of $\mathbf{G}(F_r)$. We have $\B_G=\B_G'^{Gal(F_r/F)}$ and we have an action of $\sigma$ on $\B_G'$ that commutes with $Gal(F_r/F)$. 

\textbf{Step 1:} Different $H$ orbits of $\Omega$ colored chambers in $\B_G$ remain different in $\B_G'$.

Let $\C,\C'$ be two chambers of $\B_G$, assume we have $h\in \mathbf{H}(F_r)$ such that $h\C=\C'$. We will find $h'\in H$ such that $h'\C=\C'$ and $\omega(h)=\omega(h')$. This will prove step 1.

Let $\A\subset \B_G$ be a $\sigma$ stable apartment that contains $\C$ and let $\A'\subset \B_G$ be a $\sigma$ stable apartment that contains $\C'$. There is $g\in G$ such that $\omega(g)=\omega(h)$, $g\C=\C'$ and $g\A=\A'$. The apartments $g\A$ and $h\A$ are two $\sigma$ stable apartments of $\B_G'$ that contain $g\C$. By Proposition \ref{main lemma} we may change $h$ such that we would have $h\A=g\A$.

We have $g\A=h\A$ and $g\C=h\C$, let $\mathbf{T}\subset \mathbf{G}$ be a maximal split torus such that $\mathbf{T}(F_r)$ corresponds to $\A$. Notice that $\mathbf{T}$ is split over $F$. We have $g^{-1}h\in \mathbf{T}(\s_r)$.

 $r$ is odd and $\mathbf{T}$ is split over $F$ so $\mathbf{T}(\s)/\mathbf{T}^2(\s)\rightarrow \mathbf{T}(\s_r)/\mathbf{T}^2(\s_r)$ is an isomorphism of finite groups. Thus, we may replace $g$ by an element of $\mathbf{T}(\s)$ and assume that $g^{-1}h\in \mathbf{T}^2(\s_r)$, so we can write $g^{-1}h=t_1t^{-1}_2$ with $t_1,t_2\in \mathbf{T}^2(\s_r)$, $\sigma(t_1)=t^{-1}_1$ and $\sigma(t_2)=t_2$. Changing $h$ we may assume further that $t_2=1$.

In conclusion, we have $g\in G$, $h\in \mathbf{H}(F_r)$ and $t_1\in \mathbf{T}^2(\s_r)$ with $\sigma(t_1)=t^{-1}_1$ and $g^{-1}h=t_1$. We have $\sigma(g)t_1^{-1}=\sigma(gt_1)=gt_1$ or $t^2_1=g^{-1}\sigma(g)\in G$, thus $t^2_1$ is defined over $F$. The number $r$ is odd therefore $t_1$ is also defined over $F$. Thus $h=gt_1$ is defined over $F$ and we are done with step 1. 

\textbf{Step 2:} There are no additional $\mathbf{H}$ orbits over $F_r$ that are not defined over $F$.

The $\Omega$ colors of the chambers do not depend on the field, therefore the only way a new orbit may appear is if there is a chamber in $\B_G'$ whose $\mathbf{H}(F_r)$ does not contain any chamber in $\B_G$.

Suppose such a chamber exists. There is a gallery between any two chambers of $\B_G'$, therefore we can find $\C,\C'\in \B_G'$ which share a facet of codimension 1 such that the $\mathbf{H}(F_r)$ orbit of $\C$ contains a chamber in $\B_G$, and the $\mathbf{H}(F_r)$ orbit of $\C'$ does not contain a chamber in $\B_G$. Without loss of generality, we may assume that $\C\subset \B_G$.  

 Let $f=\C\cap \C'$, let $\A$ be a $\sigma$ stable apartment of $\B_G$ that contains $\C$, and let $\C''\subset \A$ be a chamber that contains $f$ and is different from $\C$. 
 
 We must have $l_\sigma(\C)=l_\sigma(\C')$. Otherwise by Lemma \ref{3.9}, all chambers in $\B_G'$ that contain $f$ will be either in the $\mathbf{H}(F_r)$ orbit of $\C$ or in the $\mathbf{H}(F_r)$ orbit of $\C''$. But $\C'$ is not in the $\mathbf{H}(F_r)$ orbit of neither $\C$ nor $\C''$.

 Let $f'=f\cup \sigma(f)$, and let $C_f$ be the set of chambers in $\B_G'$ that contain $f$.
 
 Let $G^r_{f'}={\mathcal G}_{f'}(k_r)/R_u({\mathcal G}_{f'}(k_r))$. By Proposition \ref{orbit correspondacne} the $\mathbf{H}(F_r)\cap Ker(\omega)$ orbits on $C_f$ correspond to the $(G^r_{f'})^\sigma$ orbits on the flag variety of $G^r_{f'}$. Following the construction of this correspondence we see that it commutes with the action of $Gal(F_r/F)$. Checking all the cases of symmetric spaces in Proposition \ref{7.3} we see that there are no cases where a new orbit appears after an extension of odd rank. This completes Step 2.
 
We proved that vertices of $\Gamma_{X(F_r)}$ are the same as the vertices of $\Gamma_X$. The argument for step 2 also proves that the hyper-edges are the same. Thus we are done.
\end{proof}

\section{Examples}

In this section, we demonstrate the main results of this paper through two specific examples.

\subsection{$G=SL_2,H=T$}
Let $G=SL_2(F)$, and let $\sigma$ be conjugation by $\begin{pmatrix}
        -1 & 0  \\
        0 & 1
        \end{pmatrix}$. Then $H=G^\sigma=T$ is the torus of diagonal matrices. 
        
        Let $\B_G$ be the building of $SL_2(F)$, which is the free $q+1$ regular graph. 
        
        Let $W_{aff}=N_G(T)/T^0$, it is the free group generated by two elements of order 2, $W_{aff}=<s_0,s_1>$. Denote by $l$ the length function on the Coxeter group $W_{aff}$.

\begin{prop}\label{8.1}
Let $T_{s_0},T_{s_1}$ be the generators of $H(G,I)$ that correspond to the simple reflections $s_0,s_1\in W_{aff}$.

We have $S(X)^I\cong \mathbb{C}\oplus H(G,I)/H(G,I)(T_{s_0}+1)\oplus H(G,I)/H(G,I)(T_{s_1}+1)\oplus H(G,I)$ where  both $T_{s_1}$ and $T_{s_0}$ act on $\mathbb{C}$ by -1.
\end{prop}

This follows from the description given by Theorem \ref{generic}. The next proposition summarizes the results about the structure of $S(X)^I$ in this case. We will assume $4|q-3$. In the case $4|q-1$ the next proposition will have to be altered slightly. The difference is that in the case of $4|q-3$, the action of $H$ on the set of $\sigma$ stable maximal split tori has $3$ orbits. In the case of $4|q-1$, the same action has $5$ orbits.
        
\begin{prop}\label{8.2}
\begin{itemize}
    \item There are three orbits for the action of $T$ on the $\sigma$ stable apartments in $\B_G$. They correspond to three different $\sigma$ stable tori. 
    
    The representatives of the $\sigma$ stable tori under the action of $T$ are given by $g^{-1}Tg$ for $g$ in $\{id,\begin{pmatrix}
        1 & 1  \\
        -\frac{1}{2} & \frac{1}{2}
        \end{pmatrix},\begin{pmatrix}
        1 & \pi  \\
        -\frac{1}{2\pi} & \frac{1}{2}
        \end{pmatrix}\}$. Here $\pi$ is a uniformizer of $F$.
    \item The $I$ orbits on $X$ corresponds to $W_{aff}\mathbin{\mathaccent\cdot\cup} W_{aff}\mathbin{\mathaccent\cdot\cup} W_{aff}/T$. Denote the orbits in the first two copies of $W_{aff}$ by $x_w,y_w$ for $w\in W_{aff}$. Denote by $z_0,z_1$ the two orbits in the quotient $W_{aff}/T$.
    
    \item the length function on the orbits is given by $l_\sigma(z_0)=l_\sigma(z_1)=0$, $l_\sigma(x_w)=l_\sigma(y_w)=2\floor{\frac{l(w)}{2}}+1$.
    
    \item We have $$z_0T_{s_0}+z_0=z_1T_{s_0}+z_1=z_0+z_1+x_{s_0}+x_1$$
    
    $$z_0T_{s_1}+z_0=z_1T_{s_1}+z_1=z_0+z_1+y_{1}+y_{s_1}$$
    
    $$x_{s_0}T_{s_0}+x_{s_0}=x_{1}T_{s_0}+x_{1}=\frac{q-1}{2}(x_{s_0}+x_{1}+z_0+z_1)$$
    
    $$y_{s_1}T_{s_1}+y_{s_1}=y_{1}T_{s_1}+y_{1}=\frac{q-1}{2}(y_{s_1}+y_{1}+z_0+z_1)$$
    
    $$x_wT_s=x_{ws}\text{ if } w\notin \{s_0,1\}\text{ or }s=s_1\text{ and }l(ws)>l(w)$$
    
    $$y_wT_s=y_{ws}\text{ if } w\notin \{1,s_1\}\text{ or }s=s_0\text{ and }l(ws)>l(w)$$
    
    $$x_wT_s=qx_{ws}+(q-1)x_w\text{ if } w\notin \{s_0,1\}\text{ or }s=s_1\text{ and }l(ws)<l(w)$$
    
    $$y_wT_s=qy_{ws}+(q-1)y_w\text{ if } w\notin \{1,s_1\}\text{ or }s=s_0\text{ and }l(ws)<l(w)$$
    
    \item The hyper-graph $\Gamma_X$ is the hyper-graph appearing in Figure \ref{fig:SL_2-T} continued downwards periodically. 

    \begin{figure}
        \centering
    \includegraphics[width=0.5\linewidth]{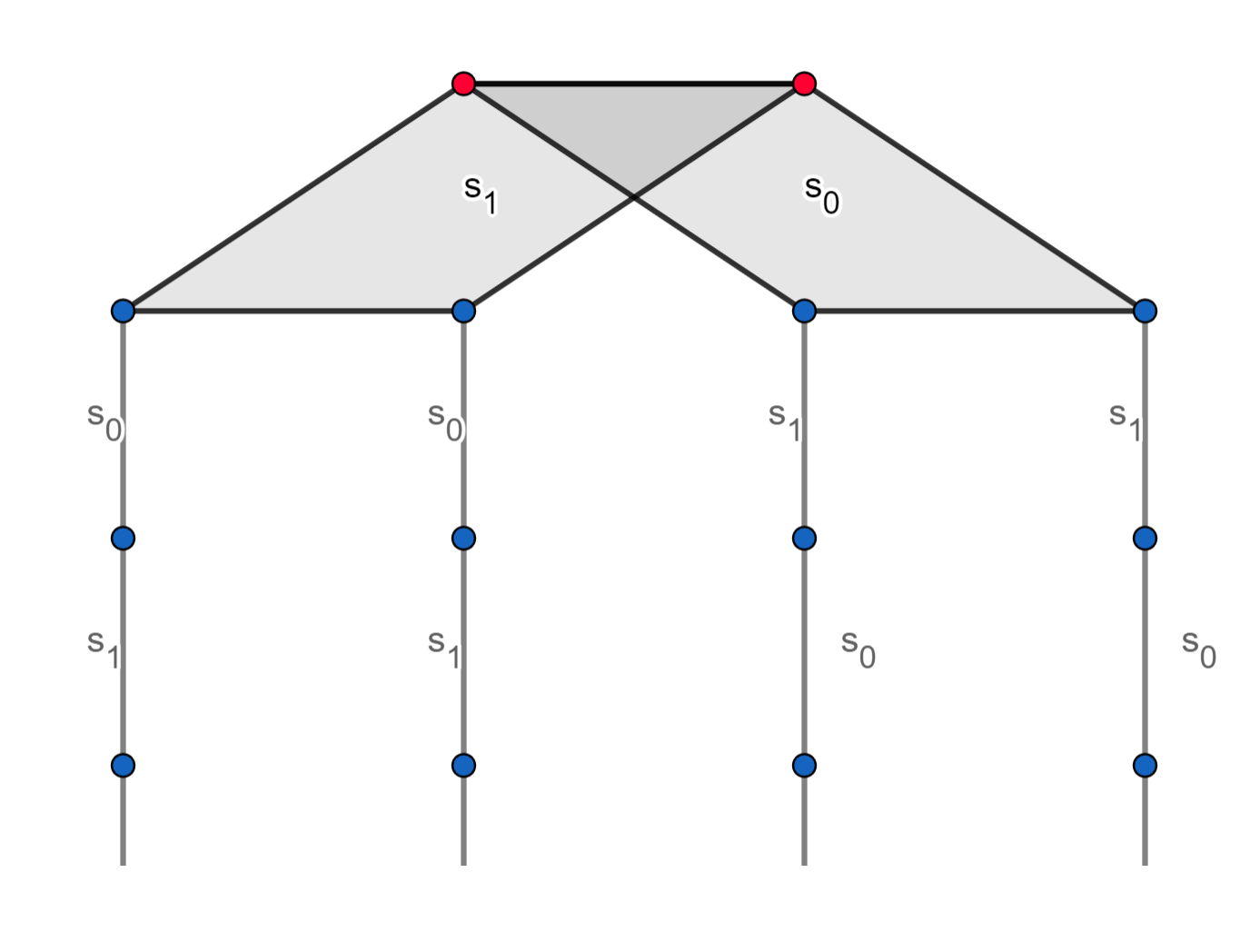}
        \caption{Part of the hyper-graph $\Gamma_X$ of the symmetric space $T\backslash SL_2$. Each parallelogram represents a hyper-edge with 4 vertices, each other hyper-edge contains exactly 2 vertices.}
        \label{fig:SL_2-T}
    \end{figure}
    
\end{itemize}

\end{prop}

Proposition \ref{8.1} follows from Proposition \ref{8.2} by taking the generators $z_0-z_1,z_0+z_1,x_1-x_{s_0}$ and $y_1-y_{s_0}$.

\subsection{$G=SL_{2n},H=Sp_{2n}$}

Let $G=SL_n(F)$, $H=Sp_{2n}(F)$ is a symmetric subgroup of $G$. The space $X=H\backslash G$ can be identified with antisymmetric matrices in $G$.

Let $T\subset G$ be the maximal torus of diagonal matrices. Embed $W_{aff}=N_G(T)/T^0$ into $G$ as the group of permutation matrices whose non zero entries are powers of the uniformizer $\pi$ (up to a sign).

\begin{prop}
\begin{itemize}
    \item There is a unique orbit for the action of $H$ on the $\sigma$ stable apartments in $\B_G$. The $I$ orbits on $X$ correspond to antisymmetric matrices whose non zero entries above the diagonal are powers of the uniformizer $\pi$. i.e. we can identify $X$ with the space of antisymmetric matrices and every $I$ orbit on $X$ contains a unique element of the mentioned form.

    Denote by $nor: AntiSymmetric\rightarrow X/I$ the map that takes an antisymmetric matrix and replaces every entry $a$ above the diagonal with $\pi^{\nu(a)}$ and every entry $a$ below the diagonal by $-\pi^{\nu(a)}$. Here, $\nu$ is the valuation of $F$.
    \item  The right action of $W_{aff}$ on $X/I$ is given by $x\times w=nor(w^txw)$.
    \item The length function $l_\sigma$ is the usual length function $l$ on $W_{aff}$ restricted to the antisymmetric matrices minus 1. $$l_\sigma(x)=l(x)-1$$

    \item The relations are the following:
    
    $$1_xT_s=1_{s^txs}\text{ if } l(s^txs)>l(x)$$
    
    $$1_xT_s=q1_{s^txs}+(q-1)1_x\text{ if } l(s^txs)<l(x)$$
    
    $$1_xT_s=q1_x\text{ if }s^txs=x$$
    
    \item In the case $n=2$, $W_{aff}$ is generated by four simple reflections $s_0,s_1,s_2$ and $s_3$. The hyper-graph $\Gamma_X$ is the graph appearing in Figure \ref{fig:SL_4-Sp_4} continued downward periodically.

    \clearpage

        \begin{figure}
        \centering
    \includegraphics[width=0.5\linewidth]{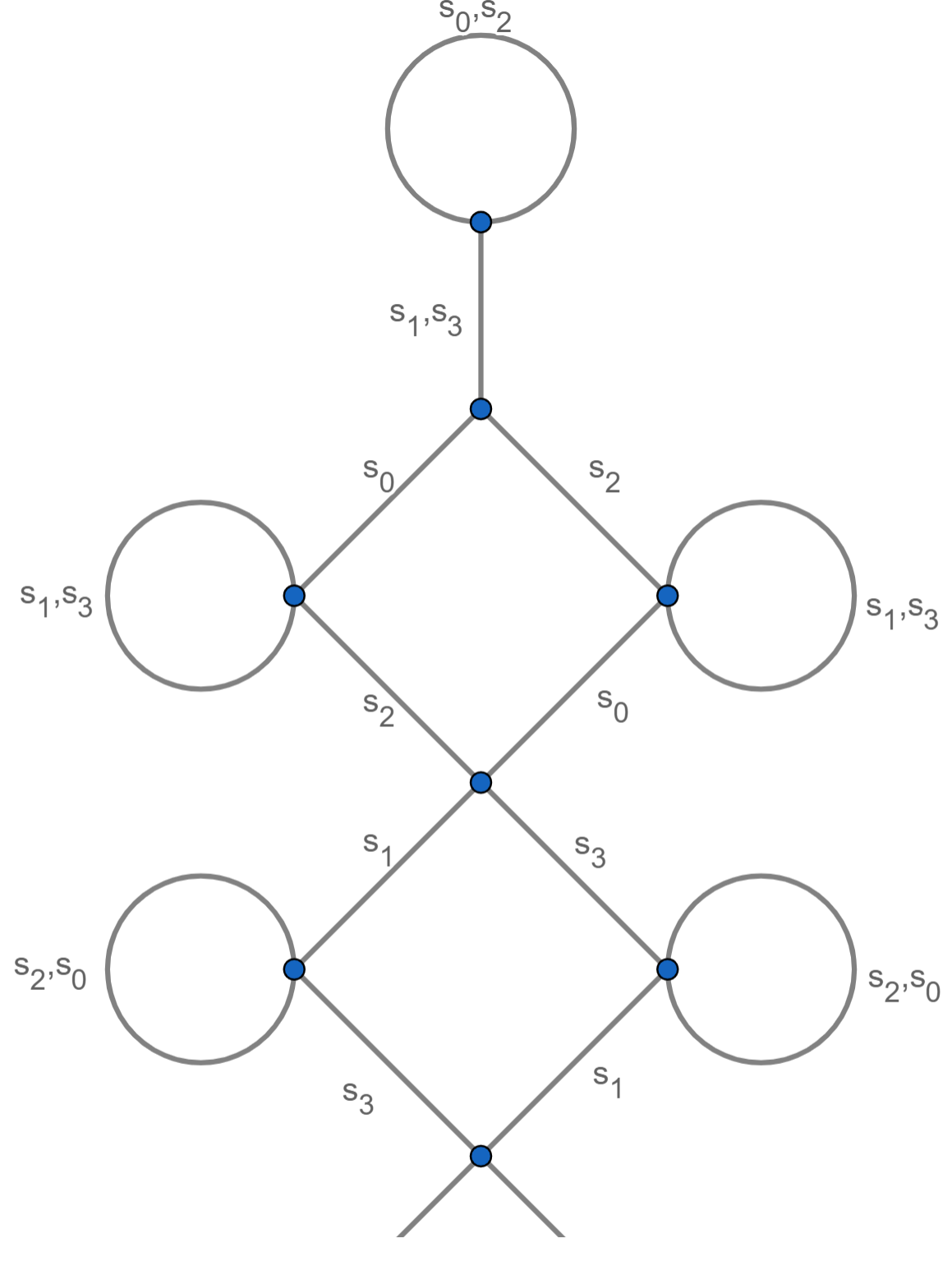}
        \caption{Part of the hyper-graph $\Gamma_X$ of the symmetric space $Sp_4\backslash SL_4$. Each hyper-edge is either of size 2 or 1.}
        \label{fig:SL_4-Sp_4}
    \end{figure}

\end{itemize}
\end{prop}

\bibliographystyle{alphaurl}
\bibliography{mybib}

\end{document}